\RequirePackage{fix-cm}
\documentclass[smallextended]{svjour3}       % onecolumn (second format)
\usepackage{graphicx}
\usepackage{amssymb,graphicx,amsmath}
\usepackage{amsopn,amstext,amsfonts}
\usepackage{color,hyperref,xcolor}
\usepackage{graphicx,pgf,tikz,epstopdf}
\usepackage{algorithm,algorithmic}
\usepackage{rotating}
\usepackage{caption}
\usepackage{subcaption}
\newtheorem{prop}{Proposition}
\newtheorem{lem}{Lemma}
\newtheorem{as}{Assumption}
\newcommand{\revision}[1]{{\color{black} #1}} 

\begin{document}

\title{Low-Rank Factorization for Rank Minimization with Nonconvex Regularizers
\thanks{This work was supported in part by National Science Foundation under Grant Number DMS-1736326.}
} 


\author{April Sagan         \and
        John E. Mitchell %etc.
}


\institute{A. Sagan \at
              \email{aprilsagan1729@gmail.com}           %  \\
           \and
           J. Mitchell \at
            \email{mitchj@rpi.edu}           %  \\
}

\date{Received: date / Accepted: date}
\newcounter{verbose}
\setcounter{verbose}{0}

\maketitle

\begin{abstract}
Rank minimization is of interest in machine learning applications such as recommender systems and robust principal component analysis.  Minimizing the convex relaxation to the rank minimization problem, the nuclear norm, is an effective technique to solve the problem with strong performance guarantees.  However, nonconvex relaxations have less estimation bias than the nuclear norm and can more accurately reduce the effect of noise on the measurements.

We develop efficient algorithms based on iteratively reweighted nuclear norm schemes, while also utilizing the low rank factorization for semidefinite programs put forth by Burer and Monteiro.  We prove convergence and computationally show the advantages over convex relaxations and alternating minimization methods.  Additionally, the computational complexity of each iteration of our algorithm is on par with other state of the art algorithms, allowing us to quickly find solutions to the rank minimization problem for large matrices.  

\keywords{Rank Minimization \and Matrix Completion \and Nonconvex Regularizers \and Semidefinite Programming }
\end{abstract}

\section{Introduction}

We consider the rank minimization problem with linear constraints formulated as 
		\begin{equation*}
		\label{eqn:rankMin}
	\begin{aligned}
	&\underset{X \in \mathcal{S}^n}{\text{min}}
	& \text{rank}(X) + \phi(X)\\
	& \text{subject to}
	& \mathcal{A}(X)=b\\
	&& X \succeq 0\\
	\end{aligned}
	\end{equation*}
where $\mathcal{S}^n$ denotes the set of symmetric $n \times n$ matrices, $\mathcal{A}: \mathcal{S}^n \rightarrow \mathbb{R}^m$ is an linear map, $b\in \mathbb{R}^m$ is the measurement vector, and $\phi(X)$ is an $L$- smooth function.  A common example is matrix completion, in which the linear constraint is $ P_\Omega(M)=P_\Omega(X)$, where $\Omega$ is the set of indices $(i,j)$ of known points in the matrix, and $P_\Omega: \mathbb{R}^{m \times n} \rightarrow \mathbb{R}^{m \times n}$ is the projection onto the set of matrices which the entry $(i,j)$ vanishes for all $(i,j) \notin \Omega$.  Formally, we define $P_\Omega$ as
\[ P_\Omega(X)_{ij}=\begin{cases} 
       0& (i,j) \notin \Omega\\
      X_{ij} & (i,j) \in \Omega
   \end{cases}
\]
Additionally, in the presence of noise, we can penalize the constraint by adding $ \phi(X)= \frac{\beta}{2} ||P_\Omega(X-M)||_F^2$ to the objective function, with a parameter $ \beta$.
Solving the rank minimization problem directly is impractical due to the rank function being non-convex and highly discontinuous.  In practice, it is common to instead minimize the convex relaxation to the rank function known as the nuclear norm, which is defined as the sum of the singular values of the matrix, or in the case of positive semidefinite matrices, the trace.  
		\begin{equation*}
		\label{eqn:nnm}
	\begin{aligned}
	&\underset{X}{\text{min}}
	& \text{trace}(X)\\
	& \text{subject to}
	& \mathcal{A}(X)=b\\
	&& X \succeq 0\\
	\end{aligned}
	\end{equation*}

The nuclear norm, denoted by $||X||_*=\sum_{i=1}^n \sigma_i(X)$ where $\sigma_i(X)$ is the $i^{th}$ singular value of $X$, is the tightest convex relaxation, and in the case of matrix completion on an $n$ by $n$ matrix known to be at most rank $r$, it has been shown to exactly recover the original matrix with high probability if at least $C n r \log(n)$ entries are observed, for an absolute constant $C$, under the assumption that the original matrix satisfies the incoherence property \cite{Candes:2010:PCR:1823677.1823678}.
{
However, minimizing the nuclear norm is not always the best approach.  As observed in the similar problem of $l_0$ norm minimization, the convex relaxation, the $l_1$ norm, introduces an estimation bias \cite{zhang2010nearly}.  Consider the following rank minimization problem:
$$ \underset{X \in \mathbb{R}^{m \times n}}{\text{min}} ||X||_*+ \frac{\beta}{2} ||P_\Omega(\tilde{M}-X)||_F^2 $$
where $\tilde{M}$ is a low rank matrix, $M$, plus Gaussian noise.  As we show in Section 2, the minimizer to the expected value of the nuclear norm regularized formulation is $\frac{p\beta}{p \beta + 1} M$, where $p= \frac{|\Omega|}{mn}$.  The bias of this formulation comes from the nuclear norm not only minimizing the smallest singular values, which correspond to the noise, but also the largest singular values, which correspond to the signal.

Another common approach to fitting a low rank matrix to a set of measurements is rank constrained optimization, wherein one attempts to find a rank $r$ matrix that minimizes an objective function. 
$$ \underset{X \in \mathbb{R}^{m \times n}}{\text{min}}  ||\mathcal{A}(X)-b||^2  \text{ subject to } \text{rank}(X)=r$$
The most common approach utilizes the low rank factorization $X=UV^T$ for $U \in \mathbb{R}^{m \times r}$ and $V \in \mathbb{R}^{n \times r}$
$$ \underset{U \in \mathbb{R}^{m \times r},V \in \mathbb{R}^{n \times r}}{\text{min}} ||\mathcal{A}(UV^T)-b||^2 $$
Because $r$ is typically much smaller than the size of the matrix, this greatly reduces the number of variables.  

In addition to finding a matrix of a given rank, this technique can be used in nuclear norm minimization as well \cite{srebro1}\cite{srebr02}\cite{JMLR:v16:hastie15a}.  The nuclear norm can be characterized as follows:
\begin{equation*}
\begin{matrix}
                  ||X||_*= &\underset{U ,V}{\text{min}}& \frac{1}{2}\big( ||U||_F^2+||V||_F^2\big) \\
     & \text{subject to} & X=UV^T
\end{matrix}
\end{equation*} 
and so, to minimize a weighted sum of the nuclear norm and a quadratic loss function, we can minimize the following
$$ \underset{U \in \mathbb{R}^{m \times r},V \in \mathbb{R}^{n \times r}}{\text{min}} \frac{1}{2}\big(||U||_F^2 +||V||_F^2 \big) + \frac{\beta}{2} ||\mathcal{A}(UV^T)-b||^2 $$

}
\subsection{Contributions}
In this paper, we consider the following general relaxation to the rank minimization
\begin{equation}
\begin{aligned}
& \underset{X}{\text{min}}
& \sum_{i=1}^n \rho(\lambda_i(X)) + \phi(X)\\
& \text{subject to}
& \mathcal{A}(X)=b\\
&& X \succeq 0\\
\label{eqn:generalizedRelaxation}
\end{aligned}
\end{equation}
where $\lambda_i(X)$ denotes the $i^{th}$ eigenvalue of $X$.   We impose the following assumptions on all $\rho$ throughout the paper.  
\begin{as} For a function $\rho: [0, \infty) \rightarrow [0, \infty)$,
\begin{enumerate}
\renewcommand{\labelenumi}{(\roman{enumi})}
\item $\rho$ is concave
\item $\rho$ is monotonically increasing
\item $\rho(0)=0$
\item For all $x \in [0,\infty)$, every subgradient of $\rho$ is finite.  Because $\rho$ is concave, it is sufficient to say
\[ \lim_{x \rightarrow 0^+} \sup_{w \in \partial \rho(x)}  w= \kappa< +\infty\]
\end{enumerate}
\label{As:concave}
\end{as}
Additionally, we may also impose one or both of the following two assumptions:
\begin{as}
The function $\rho(x)$ is strictly concave on $[0,\infty)$.
\label{As:strict}
\end{as}

\begin{as}
The function $\rho(x)$ is differentiable on $[0,\infty)$.
\label{As:differentiable}
\end{as}

\begin{table} 
--\caption{Examples of typical concave relaxations used in sparse optimization and their supergradients.  For each regularizer, $\gamma$ is a positive parameter.  For SCAD, we take $\beta>1$, and for the Schatten-$p$ norm, $0 <p\leq \revision{2}$.  Each of these functions satisfies Assumption \ref{As:concave}. }
\label{table:functions}
\centering
\begin{tabular}{c|c|c}
 &  $\rho(x)$& $ \partial \rho(x)$\\\hline
Trace Inverse\cite{Gemen} &$1-\frac{\gamma}{\gamma+x}$ &$\frac{\gamma}{(\gamma+x)^2}$ \\\hline
 Capped $l_1$ norm \cite{cappedl1}& $\text{min}({\gamma} x, 1)$ &$\begin{cases} \gamma,& x<\frac{1}{\gamma}\\ [0,\gamma] & x=\frac{1}{\gamma} \\0,& x\geq \frac{1}{\gamma}\\ \end{cases}$\\ \hline
 LogDet \cite{fazel_hindiNone} \cite{mohan_fazel2010}& $\log(x+\gamma) $ &  $\frac{\gamma}{\gamma+x}$\\\hline
 Schatten-p Norm \cite{lai_xu2013}&  $(x+\gamma)^{\frac{p}{2}}$ & $\frac{p}{2\lambda} (x+\gamma)^{\frac{p}{2}-1}$ \\ \hline
  SCAD\cite{SCAD}&
  $\begin{cases}
   \gamma x & x \leq \gamma \\
\frac{-x^2+2\gamma \alpha x-\gamma^2}{2(\alpha-1)} & \gamma \leq x \leq \alpha \gamma\\
\frac{\gamma^2(\alpha+1)}{2} & x > \alpha \gamma
\end{cases}$& 
$\begin{cases}
\gamma & x \leq \gamma\\
\frac{\alpha \gamma - x}{(\alpha-1)} & \gamma \leq x \leq \beta \gamma\\
0 & x > \alpha \gamma
\end{cases}$
\\ \hline
Laplace\cite{laplaceReg} & $1-e^{-\gamma x}$ & $\gamma e^{-\gamma x}$\\
\end{tabular}
\end{table}
{
Examples of functions meeting these assumptions that are commonly used as surrogates to the $l_0$ norm are shown in Table \ref{table:functions}. For each of the functions listed with the exception of the Shatten-$p$ norm and the LogDet relaxation, the derivative approaches 0 for large values of $x$, which would expect to greatly reduce the estimation bias.

\revision{To simplify notation, when applied to a positive semidefinite matrix, the function $\rho:\mathcal{S}_+^n\rightarrow\mathbb{R}_+$ is the sum of the regularizer $\rho$ applied to the eigenvalues of the matrix.  That is,
$$\rho(X)=\sum_i^n \rho(\lambda_i(X))$$
}

In this paper, \revision{we} show by construction that for any regularizer meeting Assumption \ref{As:concave}, the optimization problem \eqref{eqn:generalizedRelaxation} can be posed as a bi-convex optimization problem.  Our bi-convex formulation serves as an abstraction of that presented by Mohan and Fazel \cite{mohan_fazel2010}, and can be used to derive similar iterative reweighted problems.  Using our abstraction, we are able to utilize the low-rank factorization method for solving SDPs proposed by Burer and Monteiro \cite{burer_monteiro2003} in order to reduce the number of variables to $O(nr)$ where $r$ is an upper bound on the rank of the matrix, and extend the results to rectangular matrices as well. We derive algorithms based \revision{on} the low rank factorization and prove convergence.
}

\subsection{Previous Works on Nonconvex Approaches to Rank Minimization}
	In order to more closely approximate the rank of a matrix, Fazel et.\ al.\  proposed the LogDet heuristic for positive semidefinite rank minimization \cite{fazel_hindiNone}.  Instead of a convex function, the authors use the following smooth, concave function as a surrogate for the rank function.
$$\text{log}(\text{det}(X+\gamma I))= \sum_{i=1}^n \text{log}(\lambda_i(X)+\gamma)$$ 	 
where $\gamma$ is a positive parameter.
	While nonconvex, the authors put forwards a Majorize-Minimization (MM) algorithm to find a local optimum.  At each iteration, the first order Taylor expansion centered at the previous iterate is solved as a surrogate function.  The algorithm is simplified to solving the following SDP at each iteration.
		\begin{equation*}
	\begin{aligned}
	X^{(k+1)}=&\underset{X}{\text{argmin}}
	& \langle W^{(k)}, X\rangle\\
	& \text{subject to}
	& \mathcal{A}(X)=b\\
	&& X \succeq 0\\
	\end{aligned}
	\end{equation*}
where $W^{(k)} = (X^{(k-1)}+\delta I)^{-1}$.  We can view this algorithm as an iterative reweighting of the nuclear norm.  The iterative reweighted scheme was later generalized by Mohan and Fazel \cite{mohan_fazel2010} to minimize a class of surrogate functions known as the smooth Schatten-p function, defined as
 \[f_q(X)=\text{Tr}(X+\gamma I)^
\frac{p}{2}= \sum_{i=1}^n (\lambda_i(X)+\gamma)^\frac{p}{2}\]
for $0 < p \leq 2$.  The weight matrix for the Schatten-p function is  $W^{(k)}= (X^{(k-1)}+\gamma I)^{\frac{p}{2}-1}$.  Mohan and Fazel extend the algorithm for non square matrices by solving 
		\begin{equation}
		\label{eqn:MMUpdates}
	\begin{aligned}
	X^{(k+1)}=&\underset{X}{\text{argmin}}
	& \langle W^{(k)}, X^TX\rangle\\
	& \text{subject to}
	& \mathcal{A}(X)=b\\
	\end{aligned}
	\end{equation}
where $W^{(k)} = ({X^{(k-1)}}^TX^{(k-1)}+\gamma I)^{-1}$ at each iteration.   The authors prove asymptotic convergence of the iterative reweighted algorithm for $0 \leq p \leq 1$.  While this algorithm does give superior computational results, it can be very time consuming in the positive semidefinite case and will not scale well for large problems.  We show in Section 5 how this can be improved by taking advantage of the low rank property of $X$.

In recent years, many functions have been proposed as alternative non-convex surrogates to the rank function in addition to the logdet heuristic. Zhang et.\ al.\cite{Zhang_truncated_nuclear_norm} proposed minimizing the truncated nuclear norm for a general matrix $X \in \mathbb{R}^{m\times n}$, defined for a fixed constant $r$ as 

$$||X||_{r, *}= \sum_{i=r+1}^{\text{min}(m,n)} \sigma_i(X)$$
where  $\sigma_i(X)$ denotes the $i^{th}$ largest singular value.  If we consider the large singular values to represent the signal and the small singular values the noise, as in the case of noisy image reconstruction, then this minimizes only the noise.  

%The minimax concave penalty (MCP) regulatizer, defined as
%Has been proposed by Zhang \cite{zhang2010} for sparse optimization and by Wang et. al. \cite{Wang2013} for robust matrix completion.  

	The idea of minimizing a concave function of the eigenvalues has been generalized by Lu et.\ al.\ \cite{Lu2014} \cite{Lu2018}, to any monotonically increasing and Lipschitz differentiable function.  These works consider an unconstrained problem with a general loss function $\phi(X)$.  
$$ \text{min}_X \sum_{i=1}^{min(m,n)} \rho (\sigma_i(X)) +\phi(X) $$
As with the LogDet algorithm, one can derive an MM algorithm using the first order Taylor expansion about the objective function.  The authors include a proximal term.  At each iteration, the authors propose solving the following problem 

\begin{equation*}
\begin{aligned}
X^{k+1}=&\text{min} \sum_{i=1}^{\text{min}(m,n)} w_i \sigma_i(X)+\langle \nabla \phi(X^k), X-X^k\rangle +\frac{\mu}{2} ||X-X^k|| \\
=& \text{min} \sum_{i=1}^{\text{min}(m,n)} w_i \sigma_i(X)+\frac{\mu}{2}||X-Y|| 
\end{aligned}
\end{equation*}
where $Y=X^k-\nabla \phi(X^k)$ and $w_i= \rho_\gamma'(\sigma_i(X^k))$.  Much like the popular Singular Value Thresholding method put forth by Cai, Candès, and Shen \cite{cai_candes2010},  this has a closed form involving the shrinkage operator defined as $\mathcal{S}_{t}(\Sigma)=\text{Diag}(\Sigma_{ii} -t_i)_+$.  The authors prove that the subproblem has a closed form solution 
 
$$X^{k+1}=U\mathcal{S}_{\gamma w}(\Sigma) V^T$$ 
where $U\Sigma V^T$ is the singular value decomposition of $Y$.  

\revision{The shrinkage operator, however, requires computing the singular value decompositon of a possibly very large matrix, which can be time consuming and inefficent even when only the top few singular values are needed.  Similar algorithms presented by Yao et.\ al.\ address this problem by showing one only needs to find the singular value decomposition of a much smaller matrix, making the method suitable for large scale problems. \cite{yao_kwok2017}, \cite{yao_faster}}
 \\ \\

\section{Equivalent Biconvex Formulation }

It was shown by Mohan and Fazel \cite{mohan_fazel2010} that the LogDet heuristic can be reformulated as a bi-convex problem with an additional variable $W$ as follows 
\begin{equation}
\begin{aligned}
\label{eqn:linearizedinX_logDet}
& \underset{X,W}{\text{min}}
& \langle X, W \rangle +\gamma \text{trace}(W)-\log \text{det}(W) \\
& \text{subject to}
& \mathcal{A}(X)=b\\
&& X\succeq 0\\
&&  I \succeq W \succeq 0
\end{aligned}
\end{equation}
This allowed the authors to reformulate the MM algorithm outlined in equation (\ref{eqn:MMUpdates}) as an alternating method, which was of use when showing convergence of the algorithm.  We now show that an extension of this reformulation can be used for any surrogate to the rank function satisfying Assumption \ref{As:concave}.  

\begin{prop}
\label{prop:linearizeInX}
For a function $\revision{\rho}$ satisfying Assumption \ref{As:concave}, consider the following bi-convex semidefinite program
\begin{equation}
\begin{aligned}
\label{eqn:linearizedinX}
& \underset{X,W}{\text{min}}
& \langle X, W \rangle +G(W) + \phi(X) \\
& \text{subject to}
& \mathcal{A}(X)=b\\
&& X\succeq 0\\
&& \kappa I \succeq W \succeq 0
\end{aligned}
\end{equation}
where $\kappa = \sup \partial \rho (0)$, the function $G: \mathbb{S}_+^n \rightarrow \mathbb{R}$ defined as $G(W)=\sum g(\lambda_i(W))$ satisfies the following condition:
\begin{equation}
\label{eqn:derivativeInverse}
\partial g(w) \, = \, \{ -x \, : \, w \in \partial {\rho} (x) \}.
\end{equation}
Any KKT point $X^*$ of the general nonconvex relaxation (\ref{eqn:generalizedRelaxation}) can be used to construct a KKT point $(X^*, W^*)$ of (\ref{eqn:linearizedinX}) where $W^*\in \partial \revision{\rho(X^*)}$ .  Likewise, for any $(X^*, W^*)$ pair that is a KKT point of (\ref{eqn:linearizedinX}), $X^*$ is a KKT point of (\ref{eqn:generalizedRelaxation}) and  $W^*\in \partial \revision{\rho(X^*)}$.
\end{prop}

\begin{remark}
In previous works, it has been shown that the rank minimization problem \eqref{eqn:rankMin} is equivalent to the following semidefinite program with complementarity constraints:

\begin{equation}
\begin{aligned}
\label{eqn:SDPCC}
& \underset{X,U}{\text{min}}
& n-\text{trace}(U) + \phi(X) \\
& \text{subject to} 
& \langle X, U \rangle=0 \\
&& \mathcal{A}(X)=b\\
&& X\succeq 0\\
&& 0 \preceq U \preceq I
\end{aligned}
\end{equation}
Intuitively, the eigenvalues of the matrix $I-U$ are the $l_0$ norm of the eigenvalues of $X$, which implies that $n-\text{trace}(U)$ is the rank of $X$ \revision{\cite{shen_mitchell2018,sagan2020,li_qi_correlation_matrix}}.  Shen and Mitchell \cite{shen_mitchell2018} studied the problem when the complementarity constraint is relaxed as a penalty term.

\begin{equation}
\begin{aligned}
\label{eqn:SDPCC_penalty}
& \underset{X,U}{\text{min}}
&n-\text{trace}(U) + \gamma \langle X, U \rangle + \phi(X) \\
& \text{subject to} 
& \mathcal{A}(X)=b\\
&& X\succeq 0\\
&& 0 \preceq U \preceq I
\end{aligned}
\end{equation}
The penalty formulation is a biconvex semidefinite program in the form of (\ref{eqn:linearizedinX}), with $W=\frac{1}{\gamma} U$ and $G(W)= -\frac{1}{\gamma} \text{trace}(W)$. This is equivalent to the semidefinite program \eqref{eqn:generalizedRelaxation} with $\rho(x)$ being the capped $l_1$ norm, $ min(\frac{1}{\gamma} x, 1)$
\end{remark}

We want to work with the derivative of the
inverse of the derivative of $\rho(x)$, but this is only defined as stated
if $\rho_\gamma(x)$ satisfies Assumptions \ref{As:strict} and~\ref{As:differentiable}.
Under only Assumption~\ref{As:concave},
we define the function
\begin{equation}
q(t) \, := \,\inf \{ x \in [0,\infty) \, : \, t \in \partial \rho_\gamma (x) \}.
\end{equation}
Note that if $t \geq \kappa$ then $t \in \partial \rho(0)$, so $q(t)=0$ for $t \geq \kappa$.
The function $q(t)$ is defined for $t > \beta$, since $\rho(x)$ is concave;
$q(\beta)$ is also defined if $\beta$ is attained.
We let $J$ denote the domain of~$q(t)$
and $\bar{J}:=\{w \in J \, : \, w \leq \kappa\}$.
Note that $q(t)$ is lower semicontinuous;
it is continuous if Assumptions~\ref{As:strict} and~\ref{As:differentiable} hold,
in which case it is the inverse function of the derivative of $\rho(x)$
for $t \in \bar{J}$.
We can now define the function $g:J \rightarrow [0,\infty)$ as
\begin{equation}
g(w) \, := \, \int_w^{\kappa} q(t) dt.
\end{equation}

\begin{lem}   \label{lemma:convex}
The function $g(w)$ is decreasing and convex on its domain~$J$.
It is strictly convex for $w \leq \kappa$ if Assumption~\ref{As:differentiable} holds.
It is differentiable
for $w \leq \kappa$ if Assumption~\ref{As:strict} holds.
\end{lem}

\begin{lem}  \label{lemma:inverse}
For each $x \in [0,\infty)$, there exists $w \in \partial \rho(x)$ such that
\begin{equation}
-x \, \in \, \partial g (w).
\end{equation}
Further, the subdifferential is given by
\begin{equation}
\partial g(w) \, = \, \{ -x \, : \, w \in \partial \rho (x) \}.
\end{equation}
If $\rho(x)$ also satisfies Assumptions~\ref{As:strict} and~\ref{As:differentiable}
then
\begin{equation}
g' ((\rho')^{-1}(x)) \, = \, -x.
\end{equation}
\end{lem}
\ifodd\value{verbose}{
\begin{example}
Let $\rho(x)$ be the piecewise linear function
\begin{displaymath}
\rho(x) \, = \, \left\{
\begin{array}{ll}
\kappa x & \mbox{if } 0 \leq x \leq a  \\
\kappa a + \alpha (x - a) & \mbox{if } a \leq x \leq b  \\
\kappa a + \kappa (b-a) + \beta (x-b) & \mbox{if } x \geq b
\end{array}
\right.
\end{displaymath}
where $0 \leq \beta < \alpha < \kappa$ and $0 < a < b$.
Then
\begin{displaymath}
q(t) \, = \, \left\{
\begin{array}{ll}
b & \mbox{if } \beta \leq t < \alpha  \\
a & \mbox{if } \alpha \leq t < \kappa  \\
0 & \mbox{if } t \geq \kappa
\end{array}
\right.
\end{displaymath}
and
\begin{displaymath}
g(t) \, = \, \left\{
\begin{array}{ll}
b (t-\beta) + a (\alpha-\beta) & \mbox{if } \beta \leq t < \alpha  \\
a (t-\alpha) & \mbox{if } \alpha \leq t < \kappa  \\
0 & \mbox{if } t \geq \kappa
\end{array}
\right.
\end{displaymath}
Further,
\begin{displaymath}
\partial g(w) \, = \, \left\{
\begin{array}{ll}
\left(-\infty,-b\right] & \mbox{if } \beta = w \\
\left\{-b\right\} & \mbox{if } \beta < w < \alpha  \\
\left[-b,-a\right] & \mbox{if } \alpha = w  \\
\left\{-a\right\} & \mbox{if } \alpha < w < \kappa  \\
\left [-a,0\right ] & \mbox{if } w = \kappa  \\
\{ 0 \} & \mbox{if } w > \kappa
\end{array}
\right.
\end{displaymath}
\end{example}}\fi

\begin{example}
Let $\rho(x)$ be the continuous nondifferentiable function
\begin{displaymath}
\rho(x) \, = \, \left\{
\begin{array}{ll}
4x & \mbox{if } 0 \leq x \leq 2  \\
6x-x^2 & \mbox{if } 2 \leq x \leq 3  \\
9 & \mbox{if } x \geq 3
\end{array}
\right.
\end{displaymath}
which is nondifferentiable at $x=2$ and is only strictly concave for $x \in [2,3]$.
We have $\beta=0$ and $\kappa=4$.
Then
\begin{displaymath}
q(t) \, = \, \left\{
\begin{array}{ll}
3 - \frac{1}{2}t & \mbox{if } 0 \leq t \leq 2  \\
2 & \mbox{if } 2 \leq t < 4  \\
0 & \mbox{if } t \geq 4
\end{array}
\right.
\end{displaymath}
and
\begin{displaymath}
g(t) \, = \, \left\{
\begin{array}{ll}
9 + \frac{1}{4}t^2 - 3t & \mbox{if } 0 \leq t \leq 2  \\
2(4-t) & \mbox{if } 2 \leq t \leq 4  \\
0 & \mbox{if } t \geq 4
\end{array}
\right.
\end{displaymath}
Further,
\begin{displaymath}
\partial g(w) \, = \, \left\{
\begin{array}{ll}
\left[ -\infty, -3 \right] & \mbox{if } w = 0 \\
\{ \frac{1}{2}w - 3 \} & \mbox{if } 0 < w \leq 2  \\
\{ -2 \}  & \mbox{if } 2 \leq w < 4  \\
\left[ -2, 0 \right] & \mbox{if } w = 4 \\
\{0\} & \mbox{if } w > 4
\end{array}
\right.
\end{displaymath}
The lack of strict concavity on the two line segments leads to the two intervals of subgradients $\partial g(w)$ for
$w=0$ and $w=4$.
The nondifferentiability at $x=2$ leads to multiple values of $w$ having the same set of subgradients $\partial g(w)$,
namely~$\{2\}$ for $2 \leq w < 4$.
\end{example}

\noindent
{\bf Proofs of lemmas}
\begin{proof}
Proof of Lemma~\ref{lemma:convex}:

Monotonicity of $g(w)$ follows from the nonnegativity of~$q(t)$.

To show convexity, we consider $w_1<w_2$, with $w_1,w_2 \in J$,
and $0 \leq \lambda \leq 1$.
We have
\begin{displaymath}
\arraycolsep=1pt\def\arraystretch{1}
\begin{array}{rcl}
g(\lambda w_1+ &(1&-\lambda)w_2)=  \int_{\lambda w_1+ (1-\lambda)w_2}^{\kappa} q(t) dt  \\ [5pt]
&=&  \lambda \int_{w_1}^{\kappa} q(t) dt \, + \, (1-\lambda) \int_{w_2}^{\kappa} q(t) dt\, - \,  \lambda \int_{w_1}^{\lambda w_1 + (1-\lambda)w_2} q(t) dt \\
&& \qquad 
\, + \, (1-\lambda) \int_{\lambda w_1+ (1-\lambda)w_2}^{w_2} q(t) dt  \\ [5pt]
& \leq  &\lambda g(w_1) \, + \, (1-\lambda) g(w_2)  \\
&& \qquad \, - \, \lambda (\lambda w_1+ (1-\lambda)w_2 - w_1) \, g(\lambda w_1+ (1-\lambda)w_2) \\
&& \qquad \, + \, (1-\lambda) (w_2-\lambda w_1+ (1-\lambda)w_2) \, g(\lambda w_1+ (1-\lambda)w_2)  \\
&& \qquad \mbox{from monotonicity of $q(t)$}  \\  [5pt]
& =&  \lambda g(w_1) \, + \, (1-\lambda) g(w_2),
\end{array} 
\end{displaymath}
so $g(w)$ is convex.

If Assumption~\ref{As:differentiable} holds then
$q(t)$ is strictly decreasing for $\beta<w_1 \leq \kappa$,
so the inequality above holds strictly, so $g(w)$ is strictly convex.

If Assumption~\ref{As:strict} holds then $q(t)$ is continuous on~$J$,
so $g(w)$ is differentiable.
\end{proof}

\begin{proof}
Proof of Lemma~\ref{lemma:inverse}:

Since $g(w)$ is convex, the subdifferential of $g(w)$
for a slope~$w \in J$  is defined
as
\begin{displaymath}
\begin{array}{rcl}
\partial g(w)
&=& \{ \xi \, : \, \xi h \, \leq \, g(w+h) \, - \, g(w) \, \, \forall \, w+h \in J \}  \\ [5pt]
&=& \{ \xi \, : \, \xi h \, \leq \, -\int_w^{w+h} q(t) dt \, \, \forall \, w+h \in J \}  \\ [5pt]
&=& \{ \xi \, : \, \xi h \, \leq \, -h q(w+h) \, \forall \, w+h \in J \}  \\
&& \qquad \mbox{from monotonicity of $q(t)$}  \\  [5pt]
&=& \{ \xi \, : \, \xi  \, \leq \, - q(w+h) \, \forall \, h > 0, \, w+h \in J \}  \\
&& \qquad \cap \, \{ \xi \, : \, \xi  \, \geq \, - q(w+h) \, \forall \, h < 0, \, w+h \in J \}  \\ [5pt]
&=& \{ \xi \, : \, \xi  \, \leq \, - x \, \, \forall \, x \in [0,\infty)
\mbox{ with } \, w+h \in J \cap \partial f (x), \, h > 0 \}  \\
&& \qquad \cap \, \{ \xi \, : \, \xi  \, \geq \, - x \, \, \forall \, x \in [0,\infty)
\mbox{ with } \, w+h \in J \cap \partial f (x), \, h < 0 \}  \\ [5pt]
&=& \{ \, - x \, : \,  w \in  \partial f (x) \} \quad \mbox{from concavity of $\rho(x)$.} \\
\end{array}
\end{displaymath}
It follows that given $x \in [0,\infty)$, we can choose $\bar{w} \in \partial f (x)$,
and we will have $-x \in \partial g(\bar{w})$.

If Assumptions~\ref{As:strict} and~\ref{As:differentiable} hold
then $\partial \rho(x) = \{ \rho'(x) \}$,
$x$ is the unique point with derivative $\rho'(x)$,
and $g(w)$ is differentiable from
Lemma~\ref{lemma:convex}.
Setting $\bar{w}=\rho'(x)$,
the Fundamental Theorem of Calculus implies that
\begin{displaymath}
g'(\rho'(x)) \, = \, -q(\rho'(x)) \, = \, -x,
\end{displaymath}
as required.
\end{proof}

\iffalse
\begin{lem}
\label{lem:strong_convexity}
For any function $f$ satisfying Assumption \ref{As:concave}, there exists a convex function $g$ such that $g'(\rho'(x))=-x$.  Additionally, if $f$ is differentiable, then $g$ is strongly convex.
\end{lem}

\begin{proof}

We start with the case where $\rho(x)$ is differentiable and strongly concave.  Because $f$ is monotonically increasing and strongly concave, $\rho'(x)$ is strictly decreasing and thus injective. As the inverse of a strictly decreasing function is strictly decreasing, $(\rho')^{-1}(w)$ is also strictly decreasing.  Finally, the integral of a strictly decreasing function is strongly concave, which implies $g(w)$ is strongly convex.

If $f$ is not differentiable or not strongly concave, then we must consider multivalued functions.  Let $\partial \rho(x)$ be the set of subgradients of $f$ at $x$.  Because the function is concave, $\partial \rho(x)$ must be decreasing.  That is, if $x_1> x_2$, then $$\sup \partial \rho(x_1) \leq \inf \partial f (x_2).$$  The inverse function, defined as 
$$ (\partial f)^{-1} (w) = \{ x | w \in \partial f (x)\}$$
must also be decreasing.  Finally, the integral of a decreasing function is convex. 

\end{proof}
\fi
Before proving Proposition \ref{prop:linearizeInX}, we consider the following lemma.
\revision{
\begin{lem}
\label{lem:simDiag}
Let $X$ be a positive definite matrix. Let $G({W})=\sum_{i=1}^ng(\lambda_i({W}))$ be a convex function for any matrix ${W} \in \mathbb{S}^n_+$.
Let $\kappa$ be a positive constant.
If $\tilde{W}$ is a minimizer of:
\begin{displaymath}
\begin{array}{ll}
\min_{{W} \in \mathbb{S}^n_+} & \langle X, {W} \rangle \, + \, G({W})  \\
\mbox{subject to} & 0 \preceq {W} \preceq \kappa I
\end{array}
\end{displaymath}
Then $\hat{W}$ is also a minimizer with the same objective value, where $$\hat{W}=\sum_{i=1}^n \lambda_{n-i+1}(\tilde{W}) v_iv_i^T$$ and $v_i$ is the eigenvector of $X$ corresponding to the $i$th largest eigenvalue.
\end{lem}

\begin{proof}
First, note the $\hat{W}$ is a feasible point and $G(\hat{W})=G(\tilde{W})$, as the two matrices have the same eigenvalues.  

The proof relies on the Hoffman-Wielandt inequality \cite{hoffman1953}, which states that for any symmetric matrices $A$ and $B$, 
$$||A-B||_F^2 \geq ||\lambda(A)-\lambda(B)||^2$$
where $\lambda(A)$ denotes the vector of eigenvalues of $A$ in descending order.  When applied to the matrices $X$ and $-\tilde{W}$, we have
$$||X-(-\tilde{W})||_F^2 \geq \sum_{i=1}^n\big(\lambda_i(X)-\lambda_i(-\tilde{W})\big)^2=\sum_{i=1}^n\big(\lambda_i(X)+\lambda_{n-i+1}(\tilde{W})\big)^2.$$
Expanding these terms gives us the following:
    $$||X||_F^2+||\tilde{W}||_F^2 +2\langle X,\tilde{W} \rangle \geq ||\lambda(X)||^2+||\lambda(\tilde{W})||^2+2\sum_{i=1}^n \lambda_i(X)\lambda_{n-i+1}(\tilde{W})$$
    Using the fact that the Frobenius norm of a matrix is the norm of the eigenvalues, and using the simultaneous diagonalizability of $\hat{W}$ and $X$, we have:
   $$ \langle X,\tilde{W} \rangle \geq \sum_{i=1}^n \lambda_i(X)\lambda_{n-i+1}(\tilde{W})=\langle X,\hat{W} \rangle$$
So, $\hat{W}$ is a feasible point with an objective value no more than that of $\tilde{W}$, and is also a minimizer.
\end{proof}
}

Additionally, we present the technical lemma about the gradient of the objective function in (\ref{eqn:generalizedRelaxation}), which is paramount when deriving algorithms and optimality conditions.  First and second derivatives of the eigenvalue function have been studied extensively by Mangus \cite{magnus1985} and Andrew et.\ al.\ \cite{eigDerivatives1}.  
\begin{lem}
\label{lemma:derivatives}
Let $v_i$ denote the eigenvector corresponding to the $i^{th}$ eigenvalue of $X$.  If $\lambda_i(X)$ is a simple eigenvalue,
\begin{equation}
    \frac{d}{dX} \lambda_i(X) = v_i v_i^T
\end{equation}
\revision{If $\lambda_i(X)=\lambda_{i+1}(X)=\cdots=\lambda_{i+k}(X)$, then
\begin{equation*}
    \frac{d}{dX}\sum_{j=0}^k  \lambda_{i+j}(X) = \sum_{j=0}^k v_{i+j}v_{i+j}^T
\end{equation*}}

\end{lem}
Lemma \ref{lemma:derivatives} allows us to easily compute the subgradient of the objective function.
\begin{equation}
    \partial \rho(X) =\bigg\{ V \text{diag}(y_1,y_2, \ldots y_n) V^T \; \bigg| \; y_i \in \partial \rho(\lambda_i(X)) \bigg\}
\end{equation}
where $V$ denotes the matrix of eigenvectors of $X$.  We can now prove Proposition \ref{prop:linearizeInX}.
\begin{proof}
We start by considering KKT points of (\ref{eqn:linearizedinX}).  The feasible pair $(X,W)$ is a KKT point if there exists a subgradient $Z$ of $G(W)$  such that
\begin{subequations}
\begin{align}
0 &\preceq W\perp X+Z+Y\succeq 0 \label{eqn:LinearizedInXKKT_W}\\
0 &\preceq X\perp\sum_i \mu_i A_i+W + \nabla \phi(X) \succeq 0 \label{eqn:LinearizedInXKKT_X}\\
0 &\preceq Y\perp \kappa I-W\succeq 0 \label{eqn:LinearizedInXKKT_Y}
\end{align}
\end{subequations}
By Lemma \ref{lemma:derivatives}, if $W$ has eigenvectors $V^W$ and eigenvalues $w_1, w_2,\ldots, w_n$, then 
$$\partial G(W)=\bigg\{ V^W \text{diag}(z_1, z_2,\ldots, z_n) {V^W}^T \; \bigg| \; z_i \in \partial g(w) \bigg\}.$$
We start by claiming that $X$ and $W$ (and hence $Z$ and $Y$) are simultaneously diagonalizable \revision{by citing Lemma \ref{lem:simDiag}. }
\iffalse{For fixed $X=PDP^T$, problem (\ref{eqn:linearizedinX}) is equivalent to the problem
\begin{displaymath}
\begin{array}{ll}
\min_{\tilde{W} \in \mathbb{S}^n_+} & \langle D, \tilde{W} \rangle \, + \, G(\tilde{W})  \\
\mbox{subject to} & 0 \preceq \tilde{W} \preceq \kappa I
\end{array}
\end{displaymath}
where $\tilde{W}=P^TWP$, since $G(W)=G(\tilde{W})$
and
$$\langle X,W \rangle = \langle PDP^T, W \rangle = \langle D, P^TWP \rangle = \langle D, \tilde{W} \rangle.$$
From Lemma \ref{lem:simDiag}, an optimal solution to this problem has $\tilde{W}$ diagonal,
so $X$ and $W$ are simultaneously diagonalizable}\fi
Equation (\ref{eqn:LinearizedInXKKT_Y}) shows that $Y$ and $W$ are simultaneously diagonalizable.
Hence $X$, $W$, $Y$, and $Z$ are all simultaneously diagonalizable, and the KKT conditions (\ref{eqn:LinearizedInXKKT_W}) and (\ref{eqn:LinearizedInXKKT_Y}) simplify to the following.
\begin{subequations}
\begin{align}
0 &\leq \lambda_i(W)\perp  \lambda_i(X)+\lambda_i(Z)+ \lambda_i(Y)\geq 0 & \forall i=1, ..., n \label{eqn:KKT_eig_W}\\
0 &\leq \lambda_i(Y)\perp \kappa -\lambda_i(W)\geq 0& \forall i=1, ..., n \label{eqn:KKT_eig_Y}
\end{align}
\end{subequations}
If $0< \lambda_i(W)<\kappa$, we have that $ \lambda_i(Y)=0$, and so equations \eqref{eqn:KKT_eig_W} and \eqref{eqn:KKT_eig_Y} are satisfied if $\lambda_i(X)+ \lambda_i(Z)=0$. By construction of $g$ from Lemma \ref{lemma:inverse}, there exists $ w_i \in \partial \rho( \lambda_i(X))$ and $z_i\in- \partial g(w_i)$ such that $\lambda_i(Z)=z_i, \lambda_i(W)=w_i$ is a solution.

When the upper bound on the eigenvalue of $W$ is an active constraint, i.e. when $ \lambda_i(W)=\kappa$, then there exists $z_i \in \partial g(\kappa)$ such that $\lambda_i(X) +z_i\leq 0$.  Because $\partial g(\kappa)=\{0\}$, $ \lambda_i(X)=0$, which is to say $ \lambda_i(W)\in \partial \rho( \lambda_i(X))$.

Finally, we consider when $ \lambda_i(W)=0$.  Equation (\ref{eqn:KKT_eig_W})  becomes 
$$ \lambda_i(X) \geq -z_i$$
for some $z_i \in \partial g(0)$.  By Lemma \ref{lemma:inverse}, we have that $0\in \partial \rho(-z_i)$.  Because $\rho$ is concave and nondecreasing, if $0 \in \partial \rho(x_1)$, then $0 \in \partial \rho(x_2)$ for all $x_2 \geq x_1$, and so $\lambda_i(W)=0\in \partial \rho(\lambda_i(X))$.

We can now say that, in general, any $KKT$ point satisfies $\lambda_i(W)\in \partial \rho( \lambda_i(X))$ for $i=1,..,n$, and by Lemma \ref{lemma:derivatives}, $W\in \partial \rho(x)$.  The KKT conditions for (\ref{eqn:generalizedRelaxation}) state that there exists a $U \in \partial \rho(x)$ such that
$$0 \preceq X\perp\sum \mu_i A_i+U + \nabla \phi (X) \succeq 0$$
With the assignment $U=W$, it is clear that if $(X,W)$ is a KKT point of (\ref{eqn:linearizedinX}), then $X$ is a KKT point of (\ref{eqn:generalizedRelaxation}).

Conversely, consider any $X$ that is a KKT of (\ref{eqn:generalizedRelaxation}) with dual variable $\mu$.  Then, the assignment $W\in \partial \rho(x)$ and $Y=0$ satisfy (\ref{eqn:LinearizedInXKKT_X}) and (\ref{eqn:LinearizedInXKKT_Y}).  By Lemma \ref{lemma:inverse}, we have that there exists a $Z \in \partial G(W)$ such that $Z=-X$ and (\ref{eqn:LinearizedInXKKT_W}) is satisfied.
\end{proof}
Such a function is shown for various choices of nonconvex regularizers in Table \ref{table:functionsG}, and can be easily verified by showing that equation \eqref{eqn:derivativeInverse} holds.  
 We note that the function $G(W)$ is used primarily for theoretical analysis and derivation of algorithms.  In practice, one only needs the function $\rho'(x)$.

    \begin{table} 
    \caption{Function $G(W)$ and constants $\kappa$ that satisfy the conditions in Proposition \ref{prop:linearizeInX} for various concave relaxations of the rank function.}
\centering
\begin{tabular}{c|c|c|c}
 &$\partial g(w)$& $G(W)$& $\kappa$\\\hline
Trace Inverse & $ \gamma+ w^{-\frac{1}{2}} $& $\text{trace}(\gamma W -2 \sqrt{W}) $ & $ \frac{1}{\gamma}$ \\\hline
 Capped $l_1$ norm&$\frac{-1}{\gamma}$ &  $\frac{-1}{\gamma} \text{trace}(W)$& $\gamma$ \\ \hline
 LogDet& $\gamma -\frac{\gamma}{w}$ &$\gamma \text{trace}(W)-\gamma \log \text{det}(W)$& $1$\\\hline
 Schatten-p Norm& $\gamma-\big( \frac{w}{p}\big) ^\frac{1}{p-1}$ &$\text{trace}(\gamma W -\frac{2-p}{p} W^{\frac{p}{p-2}})$&$ \frac{p}{2\gamma} \gamma^{\frac{p}{2}-1}$\\ \hline
 SCAD& $-\alpha \gamma +(\alpha-1) w$ & $\text{trace}(\frac{\alpha -1}{2} W^2 - \alpha \gamma W)$& $ \gamma $\\ \hline
Laplace& $ -\frac{1}{\gamma}\log(\frac{w}{\gamma})$&$\sum_i \frac{\lambda_i(W)}{\gamma}(\log( \frac{\lambda_i(W)}{\gamma})-1)$ &$\gamma$
\\
\end{tabular}
\label{table:functionsG}
\end{table}

\subsection{Low-Rank factorization}
While the MM algorithm is efficient in the non-symmetric case, with each iteration having closed form updates which can be calculated in $O(nm^2)$ time, the algorithm is not scalable in the positive semidefinite case, as it needs to solve a semidefinite program at each iteration.  Instead, we take advantage of the low rank factorization for semidefinite programs as presented by Burer and Monteiro \cite{burer_monteiro2003} and utilized to solve the nuclear norm minimization problem by Tasissa and Lai \cite{tasissa_lai2019}.
Let $r$ be an upper bound on the rank of the matrix we seek to reconstruct.  Then, if $X$ is positive semidefinite, we have that there exists a matrix $P\in \mathbb{R}^{n \times r}$ such that $X=PP^T$. 

\begin{equation}
\label{eqn:BM_Wn}
\begin{aligned}
& \underset{P\in \mathbb{R}^{n \times r}, W \in \mathcal{S}_+^n}{\text{min}}
& \langle PP^T, W \rangle +G(W) + \phi(PP^T)\\
& \text{subject to}
& \mathcal{A}(PP^T)=b,\; \revision{0 \preceq {W} \preceq \kappa I}\\
\end{aligned}
\end{equation}

While $X$ is replaced with a variable of drastically reduced size, $W$ is left as a positive semidefinite matrix of size $n$.  To reduce the size of $W$, we propose minimizing the rank of $P^TP$ instead of the rank of $PP^T$.  

\begin{equation}
\label{eqn:BM_Wr}
\begin{aligned}
& \underset{P\in \mathbb{R}^{n \times r}, W \in \mathcal{S}_+^r}{\text{min}}
& \langle P^TP, W \rangle +G(W) + \phi(PP^T)\\
& \text{subject to}
& \mathcal{A}(PP^T)=b,\; \revision{0 \preceq {W} \preceq \kappa I}\\
\end{aligned}
\end{equation}

Intuitively, this should be equivalent due to the fact that the non-zero eigenvalues of $PP^T$ are equivalent to the nonzero eigenvalues of $P^TP$. \revision{We prove this intuition in the following proposition.}

\begin{prop}
Let $P^*\in \mathbb{R}^{n \times r}$ have the singular value  decomposition  $P^*=\sum_{i=1}^r v_i u_i^T \sigma_i^P$.  \revision{If $(P^*, W_n)$ is an optimizer of \eqref{eqn:BM_Wn}, then $$W_n^*=\sum_{i=1}^n \lambda^W_iv_i v_i^T =\sum_{i=1}^r\lambda^W_i v_i v_i^T + \kappa \sum_{i=r+1}^n v_iv_i^T,$$ and if $(P^*, W_r)$ is an optimizer of \eqref{eqn:BM_Wr}, then $W^*_r=\sum_{i=1}^r \lambda^W_i u_i u_i^T $.  Furthermore,} $(P^*, W_r^*)$ is an optimizer of \eqref{eqn:BM_Wr} if and only if $(P^*, W_n^*)$ is an optimizer of \eqref{eqn:BM_Wn}. 
%Let $P^*\in \mathbb{R}^{n \times r}$ have the singular value  decomposition  $P^*=\sum_{i=1}^r v_i u_i^T \sigma_i^P$, and let $W_n^*\in \mathbb{R}^{n \times n}$ and $W_r^*\in \mathbb{R}^{r \times r}$ be positive semidefinite matrices with eigenvalue decompositions $W_n^*=\sum_{i=1}^n v_i v_i^T \lambda^W_i$ and $W^*_r=\sum_{i=1}^n u_i u_i^T \lambda^W_i$. Then, $(P^*, W_r^*)$ is an optimizer of \eqref{eqn:BM_Wr} if and only if $(P^*, W_n^*)$ is an optimizer of \eqref{eqn:BM_Wn}.
\end{prop}

\begin{proof}
\revision{We start by proving that $W_n^*$ and $W_r^*$ have the eigenvalue decompositions stated in the proposition.  By the same reasoning as  in Proposition \ref{prop:linearizeInX}, any matrix $W\in \partial \rho(P^*{P^*}^T)$ is an optimizer to the convex semidefinite program:
\begin{equation*}
    \begin{array}{cc}
         \underset{W\in \mathcal{S}_+^n}{\text{min}}& \langle P^*{P^*}^T, W\rangle +G(W) 
         \\\text{subject to} 
         & 0 \preceq W\preceq\kappa I
    \end{array}
\end{equation*}
So, if $P^*{P^*}^T=\sum_{i=1}^r (\sigma_i^P)^2 v_iv_i^T$, then there is a minimizer $(P^*,W_n^*)$ such that $W_n^*$ has the eigendecompositon $\sum_i^r \lambda_i^W v_iv_i^T +\sum_{i=r+1}^n \kappa v_iv_i^T$, where $\lambda_i^W \in \partial \rho\big((\sigma_i^P)^2\big)$ and $\kappa=\sup \partial \rho(0)$.  Likewise, $(P^*,W_r)$ is an optimizer of \eqref{eqn:BM_Wr}, then $W^*_r=\sum_{i=1}^r \lambda_i^W u_iu_i^T$ is an optimizer, where $\lambda_i^W \in \partial \rho\big((\sigma_i^P)^2\big)$.
}

\revision{Next, we will show that if $(\Delta P, \Delta W_r)$ was a feasible descent direction in \eqref{eqn:BM_Wr} at $(P^*, W_r^*)$, then we can construct a feasible direction for \eqref{eqn:BM_Wn} at $(P^*, W_n^*)$, and vice versa.}  If $(\Delta P, \Delta W_r)$ was a feasible descent direction, then, there exists a subgradient $Z_r \in \partial G(W^*_r)$ such that 
\begin{equation}
    \revision{2}\langle P^* W^*_r +  \nabla \phi(P^*{P^*}^T) P^*, \Delta P \rangle + \langle {P^*}^T P^* + Z_r, \Delta W_r \rangle < 0.
\end{equation}
We claim that $( \Delta P, \Delta W_n)$ is a descent direction in \eqref{eqn:BM_Wn} with $$\Delta W_n= V_rU^T\Delta W_r  UV_r^T,$$ where $U \in \mathbb{R}^{r \times r}$ is the matrix whose columns are the eigenvectors of $W_r$, and $V \in \mathbb{R}^{n \times r}$ is the matrix whose columns are the first $r$ eigenvectors of $W_n$.  First note that, $P^* W_r = \sum_{i=1}^r v_i u_i \lambda_i^W \sigma_i^P=W_n^* P^*$, and $\Delta P$ is a feasible direction in \eqref{eqn:BM_Wn}.  

Next, consider the gradient of the objective of \eqref{eqn:BM_Wn} with respect to $W$, 
 \begin{align*}
     P^* {P^*}^T+ \nabla G(W^*_n)=\sum_{i=1}^r v_i v_i^T \big({\sigma_i^P} + z_i\big)+\sum_{i=r+1}^n v_i v_i^T (z_\kappa)
 \end{align*}
 where $z_i \in \partial g(\lambda_i^W)$ and $z_\kappa \in \partial g(\kappa)$.  Specifically, we chose $z_i= \lambda_i(Z_r)$, $Z_r$ be the $r$ by $r$ matrix with eigenvectors $U$ and eigenvalues $z_1, \ldots , z_r$ so that $Z_r \in G(W_r)$. 
 By Lemma \ref{lemma:inverse}, $0 \in \partial g(\kappa)$, and so the rank $r$ matrix $$Z_n:=V_r U^T \big( {P^*}^TP^*+ Z_r\big) U V_r^T$$ is a subgradient of $G$ with respect to $W_n$. Consider the inner product of the gradient of the objective of \eqref{eqn:BM_Wn} with respect to $W_n$ and the proposed descent direction for $W_n$.
\begin{align*}
    &\langle P^*{P^*}^T  + Z_n, \Delta W_n \rangle = \langle V_r U^T \big( {P^*}^TP^*+ Z_r\big) U V_r^T, \Delta W_n \rangle \\ =& \langle  \big( {P^*}^TP^*+ Z_r\big) , U V_r^T\Delta W_n V_r U^T \rangle \\= &  \langle  \big( {P^*}^TP^*+ Z_r\big) , \Delta W_r \rangle 
\end{align*}
Combining these facts gives us that $(\Delta P, \Delta W_n)$ is a descent direction:
 \begin{align*}
  &\revision{2}\langle W^*_n P^*  +  \nabla \phi(P^*{P^*}^T) P^*, \Delta P \rangle + \langle P^*{P^*}^T  + Z_n, \Delta W_n \rangle \\
    =&\revision{2}\langle P^* W^*_r +  \nabla \phi(P^*{P^*}^T) P^*, \Delta P \rangle + \langle {P^*}^T P^* + Z_r, \Delta W_r \rangle < 0
 \end{align*}
The proof of the other direction is similar.
\end{proof}

 \subsection{Extension to Nonsymmetric Matrices}

 To extend these methods to general nonsymmetric matrices $X\in \mathbb{R}^{m \times n}$, we can minimize the rank of PSD matrix $X^TX$, as done by Mohan and Fazel \cite{mohan_fazel2010}. However, this is computationally inefficient as each iteration requires finding the eigendecomposition of $X^TX$.  With this in mind, we put forth a separate extension in which we minimize the rank of the following auxiliary variable
 \begin{equation*}
 Z=\begin{bmatrix}
 G & X^T \\ X & B
 \end{bmatrix}
 \end{equation*}
It was shown by Liu et. al. that for any $X$, there exists $G$ and $B$ such that $\text{rank}(X)=\text{rank}(Z)$ and $Z\succeq 0$ \cite{doi:10.1137/090755436}. We can thus solve the following minimization problem
\begin{equation}
\begin{aligned}
& \underset{Z,W}{\text{min}}
& \langle Z, W \rangle +G(W) + \phi(X) \\
& \text{subject to}
& \mathcal{A}(X)=b\\
&& Z =\begin{bmatrix}
 G & X^T \\ X & B
 \end{bmatrix}
\succeq 0\\
&&\revision{0 \preceq {W} \preceq \kappa I}
\end{aligned}
\end{equation}
While inefficient on its own due to the matrix $W$ being $(m+n) \times (m+n)$, this formulation allows us to utilize the Burer-Monteiro approach which allowed us to efficiently solve the semidefinite case in Algorithm \ref{Agorthim:NoiseAltMin}.  We utilize the same upper bound $r$ on the rank of $X$ as before and introduce the matrix $P\in \mathbb{R}^{(m+n) \times r}$ such that $Z=PP^T$.  We decompose $P$ into $P_m$ and $P_n$ such that $P=\begin{bmatrix}
P_n\\ P_m
\end{bmatrix}$ so that $X=P_m P_n^T$. As before, we minimize the rank of $P^TP=P_m^TP_m+P_n^TP_n$. 
\begin{equation}
\begin{aligned}
& \underset{W,P_m, P_n}{\text{min}}
& \langle P_m^TP_m +P_n^TP_n, W \rangle +G(W)+ \phi(P_mP_n^T) \\
& \text{subject to}
& \mathcal{A}(P_mP_n^T)=b,\;\;\revision{0 \preceq {W} \preceq \kappa I} \\
\label{eqn:genAltMin}
\end{aligned}
\end{equation}
We note that for the special case of minimizing the nuclear norm, $W=I$, we have the well known alternating minimization method when using a quadratic loss function \cite{srebro1}\cite{srebr02}\cite{JMLR:v16:hastie15a} as follows: 
\begin{equation}
\begin{aligned}
& \underset{P_m, P_n}{\text{min}}
& ||P_n||_F^2 +||P_m||_F^2 + \frac{\beta}{2}||\mathcal{A}(P_mP_n^T)-b||^2. \\
\end{aligned}
\end{equation}

\section{Algorithms}

In most practical applications, we expect noise in our measurements, and thus an equality constraint may not be practical.  For the algorithms in this section, we restrict our focus to the problem of rank minimization with a quadratic loss function, $\phi(X)= \frac{\beta}{2} || \mathcal{A}(X)-b||_F^2$, and no linear constraints. 
Utilizing the low-rank factorization technique, \ifodd\value{verbose} {for positive semidefinite matrices, we seek to solve the following
\begin{align} 
\label{eqn:PSDnoise}
 &\underset{P, W}{\text{min}} & \langle P^TP, W \rangle +G(W)+ \frac{\beta}{2} ||\mathcal{A}(PP^T)-b||^2&:=F(P,W) \\
 & \text{subject to}
& 0 \succeq W \succeq \kappa I \nonumber&
\end{align}
Likewise,}\fi for the case of non symmetric matrices, we seek to minimize  
\begin{equation}  
 \label{eqn:rectnoise}
 \begin{array}{ll}
 \underset{X, W}{\text{min}} &\langle P_m^TP_m+P_n^TP_n, W \rangle +G(W)+ \frac{\beta}{2} ||\mathcal{A}(P_mP_n^T)-b||^2\\
  \text{subject to}\;\;\: & \revision{0 \preceq {W} \preceq \kappa I}
  \end{array}
\end{equation}

\ifodd\value{verbose} {
\subsection{PSD matrices}

In the case of PSD matrices, we solve \eqref{eqn:PSDnoise} by block coordinate descent, alternating between optimizing for $P$ and $W$.    Because the $P$ subproblem is nonconvex, a proximal term is added to ensure convergence.  The update step for $P$ is to solve $$P^k=\text{argmin}_P \langle P^TP, W^{k-1} \rangle +\frac{\beta}{2}||\mathcal{A}(PP^T) -b||_F^2 + \frac{c}{2}||P-P^{k-1}||_F^2,$$ which can be solved with gradient descent using Barzilai-Borwein step lengths, as in \cite{tasissa_lai2019}.  Note that it is unnecessary to calculate every entry of $PP^T$, as the linear operator $\mathcal{A}$ typically only requires a sparse set of entries.  The gradient, $$\nabla_PF(P,W)=PW+\beta \mathcal{A^*}(\mathcal{A}(PP^T)-b)+ c(P-P^{k-1}),$$ can be computed in $\mathcal{O}\big( r|\Omega|) \big)$ operations, where $\Omega$ is set of entries in the matrix $X$ needed in order to compute $\mathcal{A}(X)$.  In the case of matrix completion, this is simply the number of known entries in the matrix.

The update for $W$ is derived from Proposition \ref{prop:linearizeInX}, and is similar to that of other iteratively reweighted methods \cite{fazel_hindiNone} \cite{mohan_fazel2010} \cite{lai_xu2013}.
$$W^{k}=\nabla \rho({P^{k}}^T P^k)$$
Because we are minimizing the rank of the the smaller matrix $P^TP$, this update is calculated in $\mathcal{O}(r^3)$ operations.  
 \begin{algorithm}
 \caption{Alternating Minimization with General Nonconvex Regularizers (GenAltMin)  }
  \label{Agorthim:NoiseAltMin}
 \begin{algorithmic}[1]
 \renewcommand{\algorithmicrequire}{\textbf{Input:}}
 \renewcommand{\algorithmicensure}{\textbf{Output:}}
 \REQUIRE $\mathcal{A}, b$
 \ENSURE Stationary point $X$ of (\ref{eqn:rectnoise})
 \\ \textit{Initialization} :$ P^0=\text{rand}(n,r)$, $W^0=I$.
  \FOR {$k = 1,..,$}
\STATE Solve $$P^k=\text{argmin} \langle P^TP, W^{k-1} \rangle + \frac{\beta}{2} ||\mathcal{A}(PP^T)-b||^2 +\frac{c}{2} || P-P^{k-1}||_F^2$$ 
  \STATE $[V^k, \Sigma^k]=\text{eig}({P^k}^TP^k)$
  \STATE $W^{k}=V^k \rho'( \Sigma^k) {V^K}^T$
\STATE Check for Convergence
  \ENDFOR
 \RETURN $X$ 
 \end{algorithmic} 
 \end{algorithm}
The method is outlined in Algorithm \ref{Agorthim:NoiseAltMin}.

}\fi
\subsection{Alternating Methods for Rectangular Matrices}

 While the formulation for rectangular matrices could be solved by simply using Algorithm \ref{Agorthim:NoiseAltMin}, we propose an ADMM algorithm wherein we alternate over the variables $P_m$, $P_n$, and $W$.  By doing so, the subproblems in $P_m$ and $P_n$ are strongly convex.  The subproblems are as follows:

    \begin{equation}
    \begin{aligned}
    \label{eqn:Pm_update}
P_m^k&=\text{argmin}_{P_m} \langle P_m^T P_m, W \rangle +\frac{\beta}{2} ||\mathcal{A}(P_m P_n^T)-b||_F^2\\
P_n^k&=\text{argmin}_{P_n} \langle P_n^T P_n, W \rangle +\frac{\beta}{2} ||\mathcal{A}(P_m P_n^T)-b||_F^2    
   \end{aligned}
\end{equation}
The gradients of which can be calculated as 
\begin{equation*}
    \begin{aligned}
    \nabla_{P_m}F(P_m,P_n,W)&=P_m W+\beta \mathcal{A}^*(\mathcal{A}(P_mP_n^T)-b)P_n\\
        \nabla_{P_n}F(P_m,P_n,W)&=P_n W+\beta \mathcal{A}^*(\mathcal{A}(P_mP_n^T)-b)^TP_m
    \end{aligned}
\end{equation*}
\revision{where $F(P_m,P_n,W)$ is the objective function of \eqref{eqn:rectnoise}.}

\ifodd\value{verbose}{
$W$ is updated as in Algorithm \ref{}
}\else{
The update for $W$ is derived from Proposition \ref{prop:linearizeInX}, and is similar to that of other iteratively reweighted methods \cite{fazel_hindiNone} \cite{mohan_fazel2010} \cite{lai_xu2013}.
$$W^{k}=\nabla \rho({P^{k}}^T P^k)$$
Because we are minimizing the rank of the the smaller matrix $P^TP$, this update is calculated in $\mathcal{O}(r^3)$ operations.  }\fi

 \begin{algorithm}
 \caption{Alternating Minimization for Rank Minimization with a General Nonconvex Regularizer (GenAltMin)}
  \label{Agorthim:NoiseAltMin}
 \begin{algorithmic}[1]
 \renewcommand{\algorithmicrequire}{\textbf{Input:}}
 \renewcommand{\algorithmicensure}{\textbf{Output:}}
 \REQUIRE $\mathcal{A}, b$
 \ENSURE Stationary point $X$ of (\ref{eqn:rectnoise})
 \\ \textit{Initialization} :$ P^0=\text{rand}(n,r)$, $W^0=I$.
  \FOR {$k = 1,..,$}
\STATE Solve $$P_m^k=\text{argmin}_{P_m} \langle P_m^T P_m, W \rangle +\frac{\beta}{2} ||\mathcal{A}(P_m P_n^T)-b||_F^2$$ 

\STATE Solve $$P_n^k=\text{argmin}_{P_n} \langle P_n^T P_n, W \rangle +\frac{\beta}{2} ||\mathcal{A}(P_m P_n^T)-b||_F^2 $$
  \STATE $[V^k, \Sigma^k]=\text{eig}({P^k}^TP^k)$
  \STATE $W^{k}=V^k \rho'( \Sigma^k) {V^K}^T$
\STATE Check for Convergence
  \ENDFOR
 \RETURN $X$ 
 \end{algorithmic} 
 \end{algorithm} 

\ifodd\value{verbose} {
While the gradients can be calculated efficiently, this method requires solving subproblems to optimality and may not be efficient for large matrices.  A common technique to allow for closed form updates is to introduce an auxiliary variable $Z$ as follows \cite{Wen2012}
\begin{align}
     &\underset{P_n, P_m, Z, W}{\text{min}}&\langle P_m^TP_m+P_n^TP_n, W \rangle +G(W)+ \frac{\beta}{2} ||P_m^T P_n -Z||^2 \\
     &\text{subject to} & \mathcal{A}(Z)=b
\end{align}
Alternating between solving each of the four variables to optimality gives the following updates.
\begin{align}
    \label{eqn:closedFormSolutions}
    P_m^{k+1}&=Z^{k}P_n^{k}(\frac{1}{\beta} W^{k}+{P_n^{k}}^T{P_n^{k}})^{-1} \\
    P_n^{k+1}&={Z^{k}}^TP_m^{k+1}(\frac{1}{\beta}  W^{k}+{P_m^{k}}^T{P_m^{k}})^{-1} \\
    Z^{k+1}&=  P_m^{k+1}{ P_n^{k+1}}^T - \mathcal{A}^*(\mathcal{A}\mathcal{A}^*)^{-1}\big( \mathcal{A}( P_m^{k+1}{ P_n^{k+1}}^T)-b\big) \\
    W^{k+1}&=\nabla F ({P_m^{k+1}}^TP_m^  {k+1}+{{P_n^{k+1}}^T}P_n^{k+1})
\end{align}
In the case of matrix completion, the $Z$ update is simply 
$$Z^{k+1}=  P_m^{k+1}{ P_n^{k+1}}^T - \mathcal{P}_\Omega(P_m^{k+1}{ P_n^{k+1}}^T-M)$$
Because of the low-rank-plus-sparse structure of $Z$, the computational complexity of the $P_m$ and $P_n$ updates is $\mathcal{O}(r|\Omega|)$.  

While effective on its own, convergence rates are improved by utilizing successive over relaxation.  The classical successive over relaxation scheme (SOR) would be as follows:
\begin{align}
\label{eqn:SOR_updates}
    \tilde{P_m}&=Z^{k}P_n^{k}(\frac{1}{\beta} W^{k}+{P_m^{k}}^T{P_m^{k}})^{-1} \\
    P_m^{k+1}&=\omega  \tilde{P_m}+ (1-\omega) P_m^{k}\\
    \tilde{P_n}&={Z^{k}}^TP_m^{k+1}(\frac{1}{\beta} W^{k}+{P_n^{k}}^T{P_n^{k}})^{-1} \\
    P_n^{k+1}&=\omega  \tilde{P_n}+ (1-\omega) P_n^{k}\\
    \tilde{W}&= \nabla F ({P_m^{k+1}}^TP_m^  {k+1}+{{P_n^{k+1}}^T}P_n^{k+1})\\
        W^{k+1}&=\omega  \tilde{W}+ (1-\omega) W^{k}\\
    Z^{k+1}&=  P_m^{k+1}{ P_n^{k+1}}^T - \mathcal{A}^*(\mathcal{A}\mathcal{A}^*)^{-1}\big( \mathcal{A}( P_m^{k+1}{ P_n^{k+1}}^T)-b\big)
\end{align}
for $\omega \geq 1$. However, the authors of \cite{Wen2012} propose a slight deviance from the classical SOR scheme by multiplying by the term $({P_m^{k+1}}^TP_m^{k+1})^{-1}({P_m^{k+1}}^TP_m^{k})$ in the $P_n$ update.  To extend this concept to our algorithm, we propose the following SOR-like scheme.
\begin{align}
P_n^{k+1}&=\omega  \tilde{P_n}+ (1-\omega) P_n^{k} \big( ({P_m^{k+1}}^TP_m^{k+1}+\frac{1}{\beta} W^{k+1})^{-1}({P_m^{k+1}}^TP_m^{k}+\frac{1}{\beta} W^{k+1}) \big)
\end{align}
While theoretical justification for this deviance is lacking, computational results show its effectiveness in both \cite{Wen2012} and in Section 4.  

To help prevent the algorithm from stopping prematurely at bad local optimums, we initialize with a large value of $\gamma$ and decrease $\gamma$ at each iteration. In order to improve convergence rates, we also start with a low value of $\beta$ so as to quickly obtain a low rank matrix before increasing $\beta$.  The parameters are given a lower and upper bound, respectively.  Additionally, SOR schemes perform best when calculating the extrapolation term $\omega$ adaptively at each iteration, based on a trust region method.  The authors of \cite{Wen2012} propose adjusting the extrapolation term based off of the ratio of the objective function at the previous iteration to the objective function at the current iteration.  We define the ratio as 
$$ \xi=\frac{\rho_{\gamma_k}({P^{k}}^TP^k)+\beta^k||\mathcal{A}({P^{k}}^TP^k)-b||^2}{\rho_{\gamma_k}({P^{k-1}}^T P^{k-1})+\beta^k||\mathcal{A}({P^{k-1}}^TP^{k-1})-b||^2}$$
When $1 > \xi \geq 0.85$, we increment $\omega$.  When $\xi \geq 1$, this may indicate over-extrapolation, and we should reset $\omega$ to 1 and we should not accept the current iteration.  However, due to the objective function changing at each iteration, we propose a update based on filter methods\cite{filter}.  At a given iteration, we compare two additional ratios
$$ \xi_1=\frac{\rho_{\gamma_k}({P^{k}}^TP^k)}{\rho_{\gamma_k}({P^{k-1}}^T P^{k-1})}, \vspace{0.5cm} \xi_2=\frac{||\mathcal{A}({P^{k}}^TP^k)-b||^2}{||\mathcal{A}({P^{k-1}}^TP^{k-1})-b||^2}$$
If either of the ratios is less than 1, then we accept the current iteration.  The method is outlined in Algorithm \ref{alg:GenAltMin}.

  \begin{algorithm}
 \caption{ }
 \label{alg:GenAltMin_SOR}
 \begin{algorithmic}[1]
 \renewcommand{\algorithmicrequire}{\textbf{Input:}}
 \renewcommand{\algorithmicensure}{\textbf{Output:}}
 \REQUIRE $\mathcal{A}, b$
 \ENSURE Stationary point $X$ of (\ref{eqn:genAltMin})
 \\ \textit{Initialization} :$ P_n^0=\text{rand}(n,r)$, $W^0=I$.
  \FOR {$k = 1,..,$}
  \STATE Update $P_m^k, P_n^k, W^k, Z^k$ using Equation \eqref{eqn:SOR_updates}
\STATE Calculate ratios $\xi$, $\xi_1$, and $\xi_2$
\IF{$\xi \geq 1$}
\STATE $\omega=0$
\IF{$\xi_1\geq 1$ and $\xi_2 \geq 1$}
  \STATE Update $P_m^k, P_n^k, W^k, Z^k$ using Equation \eqref{eqn:SOR_updates}
  \ENDIF
\ELSIF{$\xi \geq 0.85$}
\STATE Increase $\omega$
\ENDIF
\STATE $\beta^{k+1}= \text{min}\big( 1.2\beta^k, \beta_{\text{max}})$, $\gamma^{k+1}= \text{max}\big( 0.8\gamma^k, \gamma_{\text{min}})$
  \ENDFOR
 \RETURN $X=P_mP_n^T$ 
 \end{algorithmic} 
 \end{algorithm}
 }\fi

\subsection{Alternating Steepest Descent}

  For alternating minimization without a regularizer, it has been shown computationally effective to, instead of solving subproblems to optimality, take one step in the gradient direction at each iteration \cite{Tanner2016LowRM}.  For the $P_n$ and $P_m$ updates, we can calculate the steepest descent step size.  Let $d_m$ and $d_n$ denote the gradient in the $P_m$ and $P_n$ subproblems.  Then, the steepest descent step sizes $t_m$ and $t_n$ for each subproblem respectively are can be calculated as follows
  \begin{equation*}
    \begin{aligned}
    t_m=&  \frac{\beta \langle \mathcal{A}(d_mP_n^T), \mathcal{A}(P_mP_n^T)-b\rangle +2 \langle P_m^Td_m,W\rangle}{\beta ||\mathcal{A}(d_mP_n^T)||^2+2\langle d_m^Td_m, W \rangle} \\
    t_n=&  \frac{\beta \langle \mathcal{A}(P_m d_n^T), \mathcal{A}(P_mP_n^T)-b\rangle +2 \langle P_n^Td_n,W\rangle}{\beta ||\mathcal{A}(P_m d_n^T)||^2+2\langle d_n^Td_n, W \rangle}
    \end{aligned}
\end{equation*}
Note that the step sizes can be calculated with $\mathcal{O}((m+n)r^2+r|\Omega|)$ computations.  Because solving $W$ to optimality is computationally inexpensive by comparison, we update $W$ in the same way as in Algorithm \ref{Agorthim:NoiseAltMin}.  The parameters $\beta$ and $\gamma$ are also updated in the previously mentioned way.

  \begin{algorithm}
 \caption{Alternating Steepest Descent with General Nonconvex Regularizer (GenASD)}
   \label{alg:GenASD}
 \begin{algorithmic}[1]
 \renewcommand{\algorithmicrequire}{\textbf{Input:}}
 \renewcommand{\algorithmicensure}{\textbf{Output:}}
 \REQUIRE $\mathcal{A}, b$
 \ENSURE Stationary point $X$ of (\ref{eqn:rectnoise})
 \\ \textit{Initialization} :$ P_n^0=\text{rand}(n,r)$, $W^0=I$.
  \FOR {$k = 1,..,$}
 % \STATE  $\tilde{P}_m^{k-1}=P_m^{k-1}+\frac{k-2}{k+1}(P_m^{k-1}+P_m^{k-2})$, $\tilde{P}_n^{k-1}=P_n^{k-1}+\frac{k-2}{k+1}(P_n^{k-1}+P_n^{k-2})$
  \STATE $d_m^k={P}_m^{k-1} W+\beta \mathcal{A}^*(\mathcal{A}({P}_m^{k-1}{P^{k-1}_n}^T)-b){P^{k-1}_n}$ 
  \STATE $t_m^k=\frac{\beta \langle \mathcal{A}(d_mP_n^T), \mathcal{A}(P_mP_n^T)-b\rangle +2 \langle P_m^Td_m,W\rangle}{\beta ||\mathcal{A}(d_mP_n^T)||^2+2\langle d_m^Td_m, W \rangle}$
  \STATE  $P_m^k={P}_m^{k-1}-t_m^kd_m^k$ 
    \STATE $d_n^k={P}_n^{k-1} W+\beta \mathcal{A}^*(\mathcal{A}(P_m{{P}_n^{k-1}}^T)-b)^T P^k_m$
    \STATE $t^k_n= \frac{\beta \langle \mathcal{A}(P_m d_n^T), \mathcal{A}(P_mP_n^T)-b\rangle +2 \langle P_n^Td_n,W\rangle}{\beta ||\mathcal{A}(P_m d_n^T)||^2+2\langle d_n^Td_n, W \rangle} $
  \STATE  $P_n^k={P}_n^{k-1}-t_n^kd_n^k$ 

\STATE $[V^k, \Sigma^k]=\text{eig}(P_n^TP_n +P_m^TP_m)$
  \STATE $W^{k}=V^k \rho'( \Sigma^k) {V^K}^T$
  \STATE $\beta^{k+1}= \text{min}\big( 1.2\beta^k, \beta_{\text{max}})$, $\gamma^{k+1}= \text{max}\big( 0.8\gamma^k, \gamma_{\text{min}})$
\STATE Check for Convergence
  \ENDFOR
 \RETURN $X=P_mP_n^T$ 
 \end{algorithmic} 
 \end{algorithm}

\subsection{Convergence}
Each of the algorithms presented in this section is guaranteed to converge by the main result in \cite{xu_yin2012}.  Xu and Yin show convergence of coordinated block descent algorithms to solve nonconvex optimization problems of the following form: 
\begin{equation}
\begin{aligned}
& \underset{x \in \mathcal{X}}{\text{min }}
F(x_1, \ldots , x_s) \equiv f(x_1, \ldots, x_s)+ \sum_{i=1}^s s_i(x_i)\\
\end{aligned}
\end{equation}
Denote $$f_i^k(x_i)=f(x_1^m ,\ldots, x_{i-1}^k, x_i, x_{i+1}^{k-1}, \ldots x_s^{k-1})$$ and $$\mathcal{X}_i^k(x_i)=\mathcal{X}(x_1^m ,\ldots, x_{i-1}^k, x_i, x_{i+1}^{k-1}, \ldots x_s^{k-1}).$$  Xu and Yin analyze three types of updates:
\begin{align}
    x_i^k & = \underset{x_i \in \mathcal{X}_i^k}{\text{argmin }} f_i^k(x_i)+r_i(x_i) \label{eqn:standard_update}\\ 
    x_i^k & = \underset{x_i \in \mathcal{X}_i^k}{\text{argmin }} f_i^k(x_i)+\frac{L_i^{k-1}}{2} ||x_i^{k-1}-x_i^{k-2}||^2+r_i(x_i) \label{eqn:prox_update} \\ 
    x_i^k & =\underset{x_i \in \mathcal{X}_i^k}{\text{argmin }} \langle \nabla f_i^k(\hat{x}_i^{k-1}), x_i \rangle+\frac{L_i^{k-1}}{2} ||x_i-\hat{x}_i^{k-1}||^2+r_i(x_i)\label{eqn:prox_lin_update}
    \end{align}
where $\hat{x}_i^{k-1}={x}_i^{k-1}+ w^k (x_i^{k-1}-x_i^{k-2})$, and $w^k \geq 0$ is the extrapolation weight.  

The authors assume that $F$ is continuous, bounded, and has a minimizer.  Additionally, they make assumptions on $f_i^k$ depending on the type of update used.  For the standard update \eqref{eqn:standard_update}, $f_i^k$ must be strongly convex, and for the proximal linear update \eqref{eqn:prox_lin_update}, $\nabla f_i^k$ must be $L_i^k$-Lipshitz differentiable.  For the proximal update \eqref{eqn:prox_update}, no additional assumptions are made; $f_i^k$ need not even be convex. 

\ifodd\value{verbose}{
In each of the three algorithms presented in this section, the $W$ update is solved to optimality, and thus $G(W)$ is required to be strongly convex.  As shown in Lemma \ref{lemma:convex}, this is satisfied for any differentiable regularizer satisfying Assumption \ref{As:concave}.  Algorithm \ref{Agorthim:NoiseAltMin} utilizes the proximal update for $P$ due to the block non-convexity in $P$.}\else{
In both of the algorithms presented in this section, the $W$ update is solved to optimality, and thus $G(W)$ is required to be strongly convex.  As shown in Lemma \ref{lemma:convex}, this is satisfied for any differentiable regularizer satisfying Assumption \ref{As:concave}.
}\fi

In Algorithm \ref{Agorthim:NoiseAltMin}, we utilize the standard update, and so our objective function must be strongly convex.  Because the quadratic loss function is block convex in both $P_m$ and $P_n$, it typically samples a small portion of the matrix and will not be strongly convex.  However, the terms $\langle P_m^TP_m^T, W^k \rangle$ and $\langle P_n^TP_n^T, W^k \rangle$ are strongly convex so long as $W^k$ is full rank.  Assumption \ref{As:strict} is then necessary to ensure convergence, as strong concavity in $\rho$ ensures $\rho$ is strictly increasing and that that $0 \notin \partial \rho(x)$ for any finite $x$.  

Lastly, because $\nabla_{P_m} F$ and $\nabla_{P_n} F$ are linear, Algorithm \ref{alg:GenASD} converges. 

While the capped $l_1$ norm is non-differentiable, meaning none of the algorithms in this section are guaranteed to converge when using it as the regularizer, one can modify the algorithms slightly so that it does converge as in Shen and Mitchell \cite{shen_mitchell2018}.  The authors utilize the proximal linear update for $W$ as follows:

$$W^{k+1}=\underset{{ 0 \preceq W  \preceq I}}{\text{proj}}\bigg( W^k+ \frac{1}{L^k}(X^{k+1}+\gamma I) + w^k(W^k-W^{k-1})\bigg)$$
When this update is used in any of the algorithms in this section, convergence is guaranteed without assuming differentiablity of the regularizer.

\section{Numerical Results}
Algorithms \ref{Agorthim:NoiseAltMin} and \ref{alg:GenASD} were implemented in MATLAB R2018b, and the source code to run the algorithms and reproduce every result in this section is publicly available at \url{github.com/april1729/GenAltMin}.  The numerical experiments were conducted on a Dell Laptop running Windows 10 with 16 GB of ram and an Intel Core i3-4030U CPU @ 1.90 GHz.
\ifodd\value{verbose} {
\subsection{Euclidean Distance Geometry}
The Euclidean distance problem is, given an incomplete distance matrix between points in $\mathbb{R}^d$, where $D_{ij}$ is the distance squared between two points $i$ and $j$, we want to reconstruct the location of each point.  If the location of the points is given by the matrix $P\in \mathbb{R}^{n \times d}$, we hope to find a matrix $B=PP^T$ such that at each $(i,j)$ where $D_{ij}$ is known, $D_{ij}=B_{ii}+B_{jj}-2B_{ij}$.  We also know that the rank of $B$ is $d$, and so $B$ can be recovered by the optimization problem

\begin{equation}
\begin{aligned}
& \underset{B\in \mathcal{S}_+^n}{\text{min}}
& \text{rank}(B)\\
& \text{subject to}
& D_{ij}=B_{ii}+B_{jj}-2B_{ij} \forall (i,j) \in \Omega 
\end{aligned}
\end{equation}

The experiment is run on a data set consisting of 1000 randomly generated points on a sphere  with symmetric Gaussian noise added to each entry of the distance matrix independently. The reconstructions are shown in Figure \ref{fig:spheresWithNoise}.  Note that with 0.125\% of the data known, the reconstructed shapes for each of the nonconvex relaxations is in fact 2 dimensional, as the recovered Gram matrix was rank 2. Each of the nonconvex relaxations successfully recovers the sphere with at lower percentages of known data than the convex relaxation.  

Table \ref{table:errorNoise} shows the RFNE for the reconstructed points.  The error for each of the nonconvex relaxations is similar, and each of them out perform the nuclear norm. We note that the two regularizers that level off and are expected to give less bias towards smaller matrices, SCAD and the trace inverse regularizer, slightly outperform the LogDet and Schatten $p$ norm regularizers.

\begin{figure}
\centering
\begin{tabular}{cccccc}
                  & 0.125  & 0.25   & 0.5     & 1\%     & 2\%      \\ \hline
Nuclear Norm& \includegraphics[width=0.12\columnwidth]{spheresNoise/m1p1}&  \includegraphics[width=0.12\columnwidth]{spheresNoise/m1p2}& \includegraphics[width=0.12\columnwidth]{spheresNoise/m1p3} & \includegraphics[width=0.12\columnwidth]{spheresNoise/m1p4} & \includegraphics[width=0.12\columnwidth]{spheresNoise/m1p5}\\
SCAD & \includegraphics[width=0.12\columnwidth]{spheresNoise/m2p1}&  \includegraphics[width=0.12\columnwidth]{spheresNoise/m2p2}& \includegraphics[width=0.12\columnwidth]{spheresNoise/m2p3} & \includegraphics[width=0.12\columnwidth]{spheresNoise/m2p4} & \includegraphics[width=0.12\columnwidth]{spheresNoise/m2p5}\\
Trace Inverse& \includegraphics[width=0.12\columnwidth]{spheresNoise/m3p1}&  \includegraphics[width=0.12\columnwidth]{spheresNoise/m3p2}& \includegraphics[width=0.12\columnwidth]{spheresNoise/m3p3} & \includegraphics[width=0.12\columnwidth]{spheresNoise/m3p4} & \includegraphics[width=0.12\columnwidth]{spheresNoise/m3p5}\\
Log Det& \includegraphics[width=0.12\columnwidth]{spheresNoise/m4p1}&  \includegraphics[width=0.12\columnwidth]{spheresNoise/m4p2}& \includegraphics[width=0.12\columnwidth]{spheresNoise/m4p3} & \includegraphics[width=0.12\columnwidth]{spheresNoise/m4p4} & \includegraphics[width=0.12\columnwidth]{spheresNoise/m4p5}\\
Schatten 1/2 norm& \includegraphics[width=0.12\columnwidth]{spheresNoise/m5p1}&  \includegraphics[width=0.12\columnwidth]{spheresNoise/m5p2}& \includegraphics[width=0.12\columnwidth]{spheresNoise/m5p3} & \includegraphics[width=0.12\columnwidth]{spheresNoise/m5p4} & \includegraphics[width=0.12\columnwidth]{spheresNoise/m5p5}\\
\end{tabular}
\caption{Spheres reconstructed from noisy and incomplete distance measurements via Algorithm \ref{Agorthim:NoiseAltMin}}
\label{fig:spheresWithNoise}
\end{figure}

\begin{table} 
\caption{RFNEs for points of sphere reconstructed from incomplete and noisy distance measurements via Algorithm \ref{Agorthim:NoiseAltMin}}
\label{table:errorNoise}
\centering
\begin{tabular}{|l|l|l|l|l|l|}
\hline
                  & 0.125  & 0.25   & 0.5     & 1\%     & 2\%      \\ \hline
Nuclear Norm      & 0.7148 & 0.6182 & 0.3217  & 0.07563 & 0.01778  \\ \hline
SCAD              & 0.775  & 1.019  & 0.05256 & 0.01027 & 0.006512 \\ \hline
Trace Inverse     & 0.9783 & 0.7633 & 0.04556 & 0.01029 & 0.006518 \\ \hline
LogDet            & 0.8455 & 0.706  & 0.0646  & 0.01171 & 0.006745 \\ \hline
Schatten 1/2 norm & 0.7742 & 0.6656 & 0.07106 & 0.01265 & 0.006900 \\ \hline
\end{tabular}
\end{table}

\begin{table} 
\caption{Approximate rank of the Gram matrix recovered from Algorithm \ref{Agorthim:NoiseAltMin}, calculated as the number of eigenvalues larger than 0.001}
\label{table:rankNoise}
\centering
\begin{tabular}{|l|l|l|l|l|l|}
\hline
                  & 0.125 & 0.25 & 0.5 & 1\% & 2\% \\ \hline
Nuclear Norm      & 10    & 10   & 10  & 10  & 10  \\ \hline
SCAD              & 2     & 5    & 4   & 3   & 3   \\ \hline
Trace Inverse     & 2     & 3    & 3   & 3   & 3   \\ \hline
LogDet            & 2     & 3    & 3   & 3   & 3   \\ \hline
Schatten 1/2 norm & 2     & 3    & 3   & 3   & 3   \\ \hline
\end{tabular}
\end{table}

\begin{table} 
\caption{Values of the fourth largest eigenvalue of the recovered Gram matrix for various different relaxations and percentages of data.  The Gram matrix is expected to be rank 3.  }
\label{table:4theignoise}
\centering
\begin{tabular}{|l|l|l|l|l|l|}
\hline
                  & 0.125      & 0.25       & 0.5        & 1\%       & 2\%       \\ \hline
Nuclear Norm      & 56.49      & 72.63      & 39.50      & 7.287     & 1.543     \\ \hline
SCAD              & -3.231e-14 & 4.292      & 0.1133     & 8.642e-10 & 4.412e-09 \\ \hline
Trace Inverse     & -6.337e-14 & -4.829e-14 & -2.087e-13 & 3.058e-11 & 2.029e-10 \\ \hline
LogDet            & -4.021e-14 & -7.955e-14 & -1.075e-13 & 1.107e-11 & 6.671e-09 \\ \hline
Schatten 1/2 norm & -2.803e-14 & 9.684e-14  & -2.008e-13 & 4.600e-11 & 1.797e-08 \\ \hline
\end{tabular}
\end{table}
Tables \ref{table:rankNoise} and \ref{table:4theignoise} show the rank of the recovered Gram matrix and the fourth eigenvalue of the recovered Gram matrix.  Both tables show that for 0.25\% and 0.5\% of the data, SCAD fails to recover a rank 3 Gram matrix.  We conjecture that the lack of gradient information for large eigenvalues leads to the algorithm stopping at worse local minimum.  With this in mind, these results suggest that the trace inverse regularizer is the more robust of the regularizers with an upper bound.

Table \ref{table:4theignoise} shows the fourth eigenvalue of the Gram matrix recovered by the nuclear norm to be steadily decreasing, however even more data would be needed to exactly reconstruct the points. One could remedy this by adjusting the parameters to put more weight on the regularizer than the loss function, however, this would further ...?

We apply Algorithm \ref{Agorthim:NoiseAltMin} to the Euclidean distance geometry problem for reconstructing molecular conformation of 6 different proteins from CITATION.  We sample 5\% of the distances and add 20\% noise, and then utilize Algorithm 1 with with the trace inverse regularizer and the nuclear norm regularizers.  The results are shown in Table \ref{table:4theignoise}. In addition to being marginally faster in most cases, the trace inverse regularizer leads to more accurate results.
\begin{table} 
\caption{RFNE and run-time for Algorithm \ref{alg:GenAltMin} using the Trace Inverse regularizer and Nuclear Norm regularizer on six different protein structures.}
\centering
\begin{tabular}{l|ll|ll}
\label{table:protien}
Protein & TI Error  & TI Time  & Nuclear Norm Error & Nuclear Norm time \\ \hline
1AX8    & 0.0063209 & 2.3743   & 0.019039           & 3.6334            \\
1KDH    & 0.0019397 & 45.54    & 0.0030001          & 70.9233           \\
1HOE    & 0.0070767 & 0.75863  & 0.057482           & 1.0913            \\
1BPM    & 0.0014497 & 107.8874 & 0.0019605          & 151.9404          \\
1PTQ    & 0.030404  & 0.38344  & 0.11362            & 0.37811           \\
1RGS    & 0.0033729 & 14.9117  & 0.0044945          & 35.6796          
\end{tabular}
\end{table}

}\fi

\subsection{Synthetic Data for Rectangular Matrices }

We now test Algorithms \ref{Agorthim:NoiseAltMin} and \ref{alg:GenASD} utilizing synthetically generated low rank matrices with additive Gaussian noise.  Throughout this section, we generate a matrix of size $m$ by $n$ with rank $r$ and noise parameter $d$ by the following Matlab command: 
\begin{verbatim}
          M = randn(m,r) * randn(r,n) + d * randn(m,n)
\end{verbatim}

Figures \ref{fig:percents_small} and \ref{fig:percents_large} show the Relative Frobenius Norm Error (RFNE) of the solution recovered by the nuclear norm and by the trace inverse regularizer with varying percentages of known data, along with the relative Frobenius norm of the noise matrix as a baseline.  We plot these results for a 300 by 200 matrix and a 1000 by 500 matrix, each averaged over 10 randomly generated instances.  In both figures, the trace inverse is able to outperform the baseline when only 20\% of the data is available.  Note that in each case, the trace inverse regularizer outperforms the nuclear norm.

\begin{figure*}[t!]
    \centering
    \begin{subfigure}[b]{0.45\textwidth}
        \centering
        \includegraphics[width= \columnwidth]{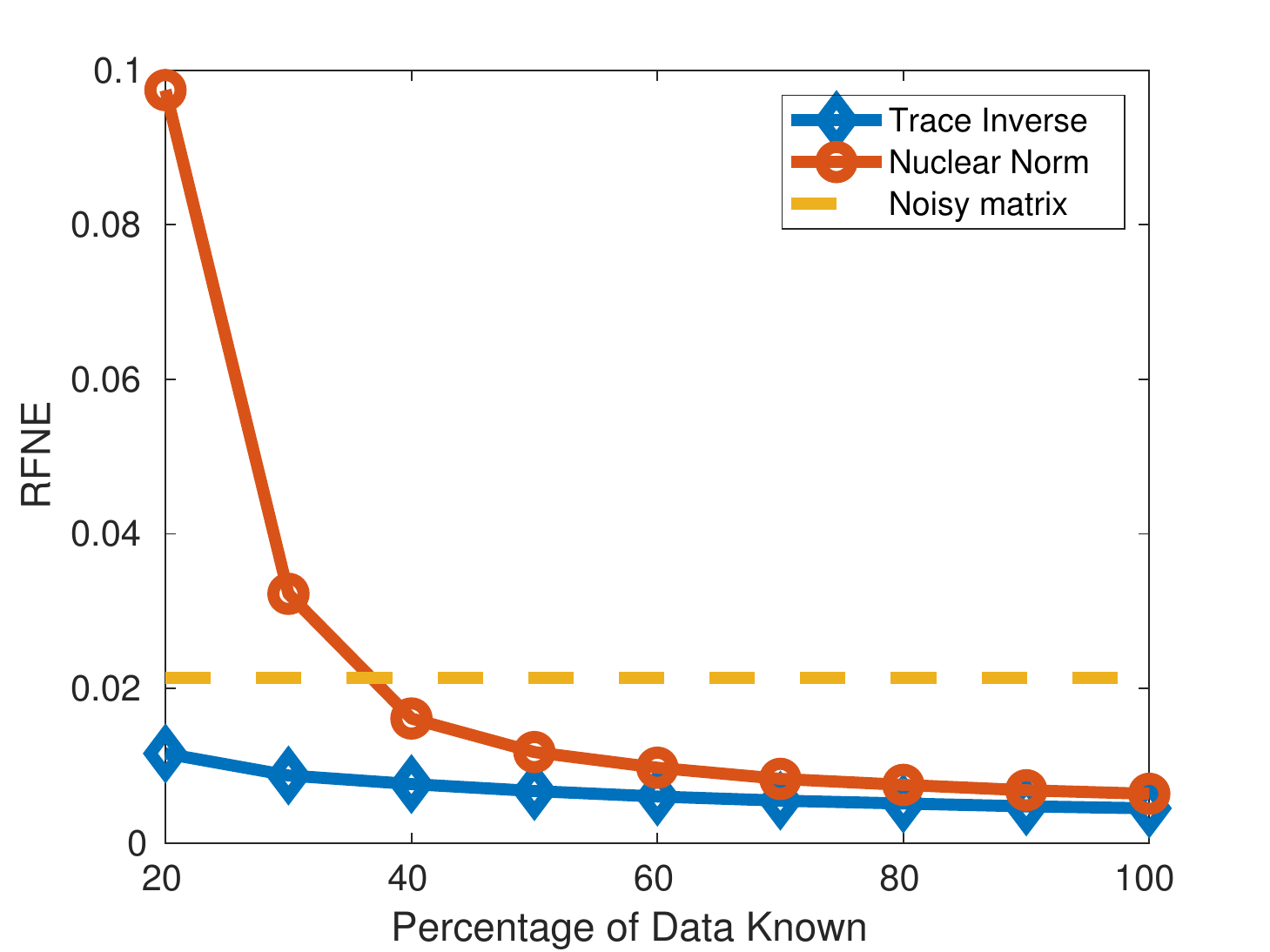}
        \caption{$m=300, n=200, r=5, d=0.05$}
                \label{fig:percents_small}

    \end{subfigure}%
    ~ 
    \begin{subfigure}[b]{0.45\textwidth}
        \centering
        \includegraphics[width= \columnwidth]{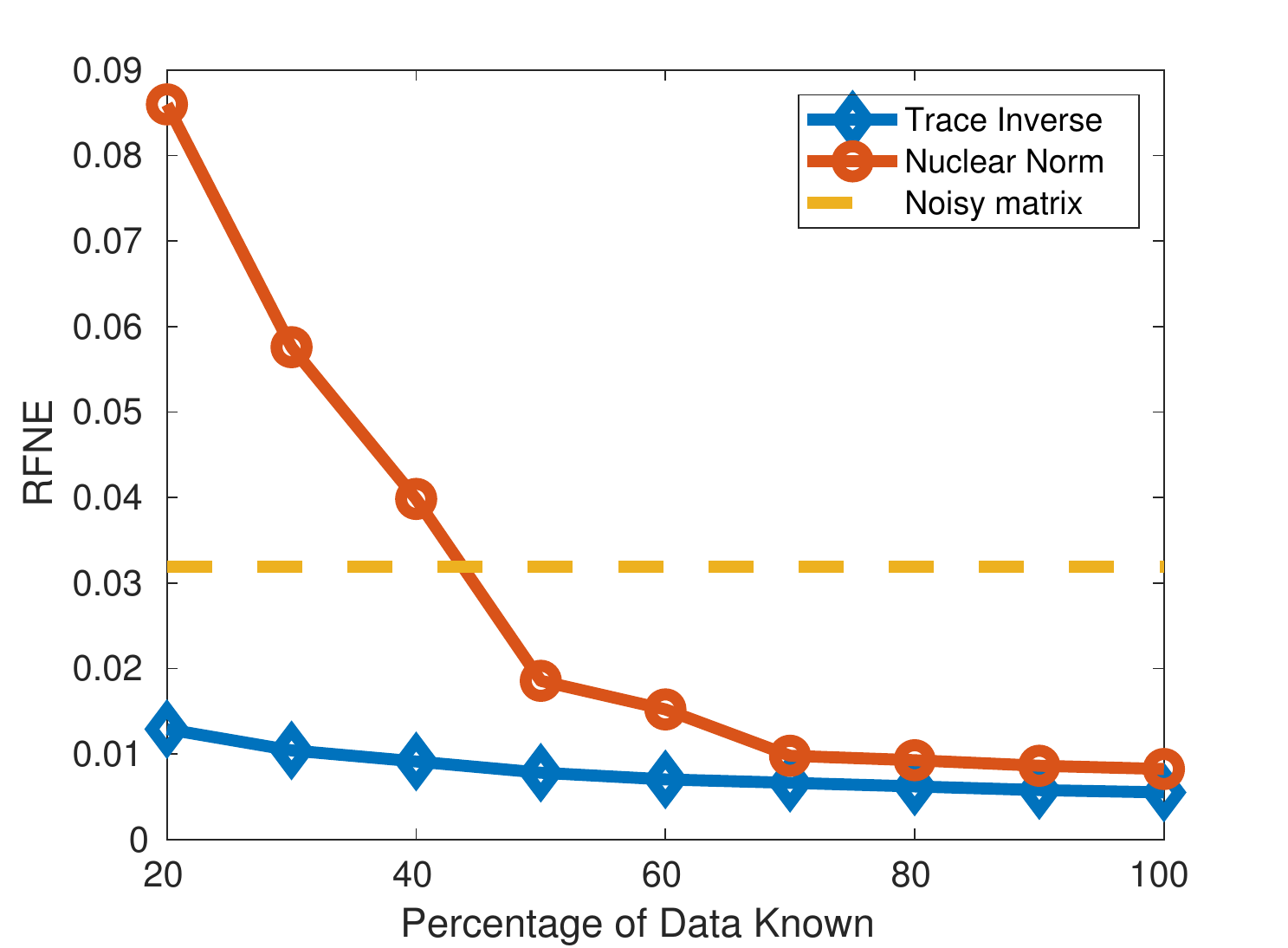}
        \caption{$m=1000, n=500, r=10, d=0.1$}
        \label{fig:percents_large}
    \end{subfigure}
    \caption{RFNE of the matrix recovered from Algorithm \ref{alg:GenASD} using both the nuclear norm and trace inverse regularizer for varying amounts of data known, along with the RFNE of the noise.}
\end{figure*}

To show that the superiority of the nonconvex regularizer is not just for certain choices of $\beta$, we show how each method performs for values of $\beta$ between $10^{-3}$ and 10 for the smaller problem and $10^{-4}$ and 1 for the larger problem in figures \ref{fig:beta_small} and \ref{fig:beta_large} respectively.  When the parameter is differed by an orders of magnitude, the results for the trace inverse regularizer are hardly affected, while the accuracy of the optimal solution to the nuclear norm problem varies a significant amount.  In fact, every value of $\beta$ for the trace inverse regularizer outperformed the optimal value of $\beta$ for the nuclear norm regularizer. 

\begin{figure*}[t!]
    \centering
    \begin{subfigure}[b]{0.45\textwidth}
        \centering
        \includegraphics[width= \columnwidth]{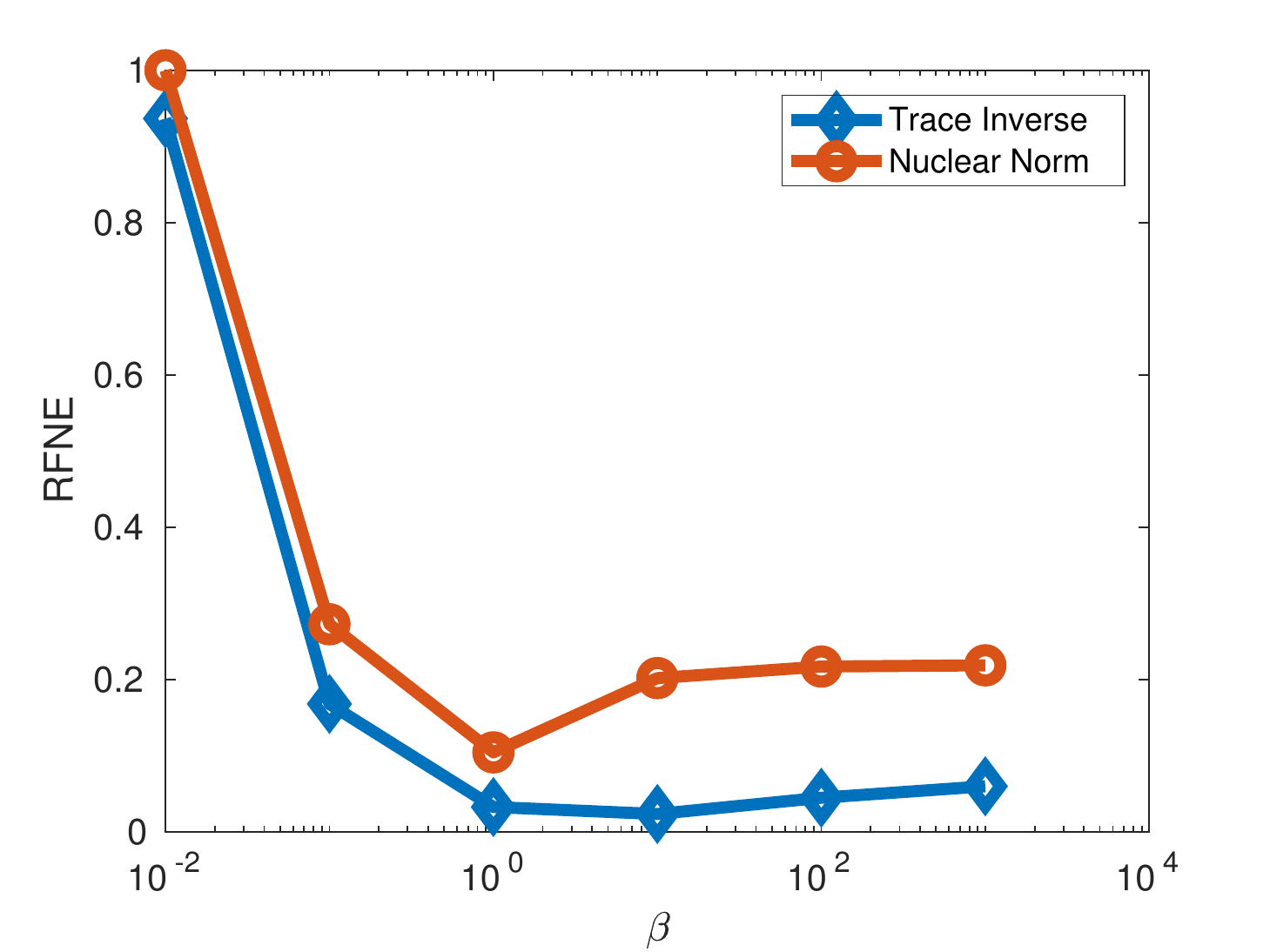}
        \caption{$m=300, n=200, r=5,$\\$ d=0.05, p=0.2$}
                \label{fig:beta_small}
    \end{subfigure}%
    ~ 
    \begin{subfigure}[b]{0.45\textwidth}
        \centering
        \includegraphics[width= \columnwidth]{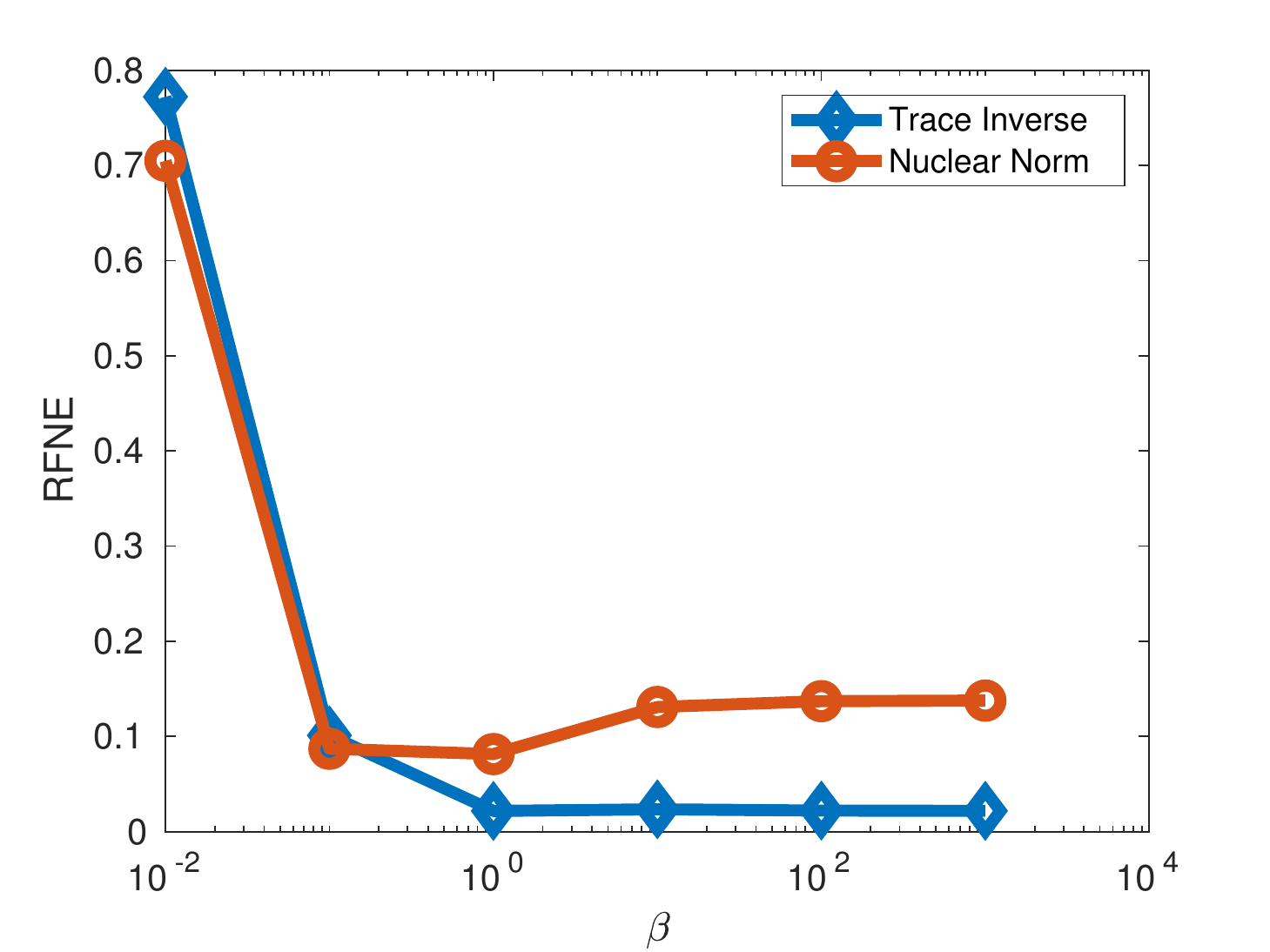}
        \caption{$m=1000, n=500, r=10,$\\$ d=0.1, p=0.2$}
        \label{fig:beta_large}
    \end{subfigure}
    \caption{RFNE for varying amounts of data known for Algorithm \ref{alg:GenASD} for both the nuclear norm and trace inverse regularizer, along with the RFNE of the noise.  The two figures show the results for different trade off parameters.  }
\end{figure*}

In order to illustrate the estimator bias of the nuclear norm formulation compared to nonconvex approaches, we plot the singular values of the reconstructed matrix utilizing both the trace inverse regularizer and the nuclear norm, along with the singular values of the original matrix.  We show this plot for varying values of $\beta$ of for a 300 by 200 matrix with rank 5 in Figure \ref{fig:svd_Small}\ifodd\value{verbose}{ and a 1000 by 500 matrix with rank 10 in Figure \ref{fig:svd_large}}\fi.  We plot the first $r$ singular values and the next $r$ singular values on a different scale, where $r$ is the rank of the matrix being recovered.  

\ifodd\value{verbose}{ We observe a similar trend in the singular values of the matrix reconstructed with the nuclear norm in both cases.  }\fi For values of $\beta$ that are smaller than 0.01, the solution is the zero matrix, and for values of $\beta$ larger than 0.1, the solution is not the correct rank.  As expected, there is a very small range in which we obtain a matrix with the correct rank.  Additionally, when the nuclear norm algorithm gives a matrix with the correct rank, the singular values reconstructed using the nuclear norm are noticeably smaller.  This is due to the fact that the nuclear norm puts equal weight on minimizing each singular value, including the ones that should not be zero. So, by increasing $\beta$, the singular values that are supposed to be zero become larger, and by decreasing $\beta$, the singular values that are not supposed to be zero become too small. 

By contrast, the top $r$ singular values for the matrix reconstructed with Algorithm $\ref{alg:GenASD}$ are approximately equal to the singular values of the original matrix.  For values of $\beta$ less than 0.01 in the first case and 0.001 in the second case, the solution to the trace inverse formulation is the correct rank.  As opposed to the convex relaxation, the nonconvex method has a sufficiently large range of $\beta$ that give a matrix of the correct rank.

\begin{figure}
    \centering
    \includegraphics[width=\columnwidth]{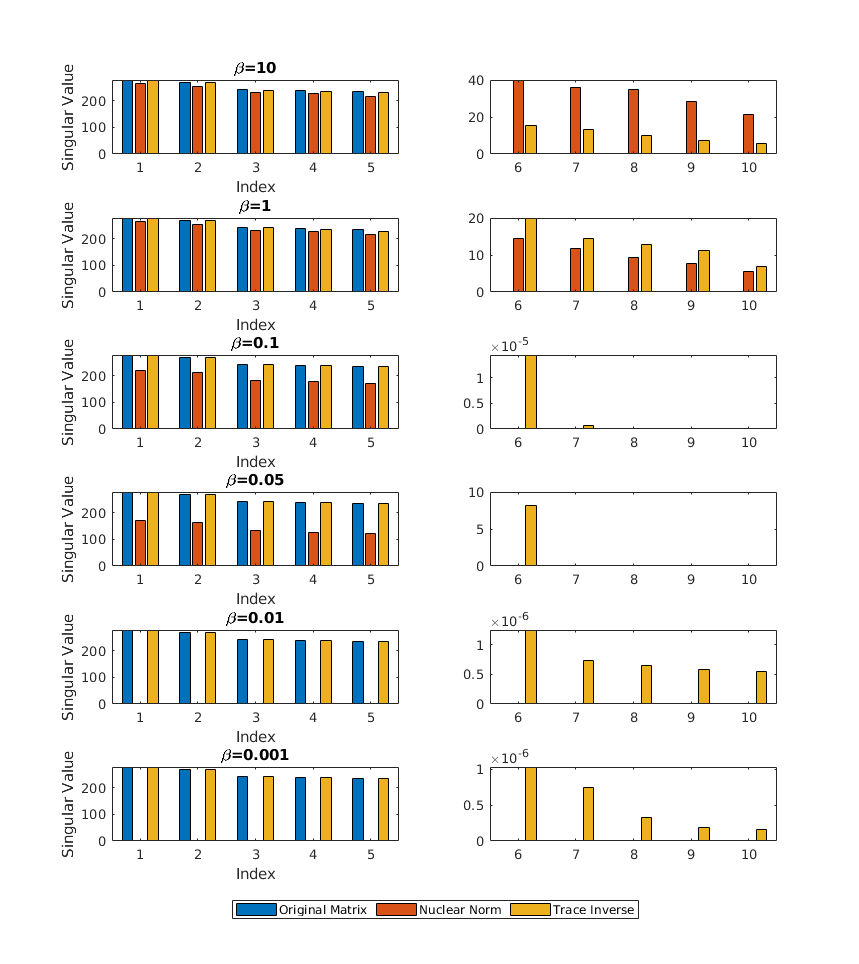}
    \caption{Singular value distribution for the matrices recovered utilizing Algorithm \ref{alg:GenASD} with the trace inverse regularizer and nuclear norm regularizer with $m=300 , n=200, r=5, d=0.05$, and $p=0.2$.}
    \label{fig:svd_Small}
\end{figure}

\ifodd\value{verbose}{
\begin{figure}
    \centering
    \includegraphics[width=\columnwidth]{svd_plot_large.png}
    \caption{Singular value distribution for the matrices recovered utilizing Algorithm \ref{alg:GenASD} with the trace inverse regularizer and nuclear norm regularizer with $m= 1000, n=500, r=10, d=0.1$, and $p=0.1$.}
    \label{fig:svd_large}
\end{figure}
}\fi

While this shows that the nonconvex formulations are significantly more robust to the choice of $\beta$, one may wonder if the added parameter controlling the curvature of the regularizer, $\gamma$, may contribute to more variability with parameter choices.  Figure \ref{fig:gamma} shows the RFNE for choices of $\gamma$ distributed between 0.03125 and 256.  Surprisingly, the figure shows that for a large range of choices of $\gamma$, the results are identical.  It is only at $\gamma=0.125$ that the nonconvex formulation loses the stability it usually has.  This behavior is expected due to the fact that the trace inverse regularizer converges to the rank function as $\gamma$ approaches 0.  For values of $\gamma$ larger than the smallest non-zero singular value of the original low rank matrix (roughly 200), the trace inverse formulation behaves more similarly to the nuclear norm, which one could also expect as the derivative of the nonconvex regularizer is approximately a constant for large values of $\gamma$.

\begin{figure*}[t!]
    \centering
    \begin{subfigure}[b]{0.45\textwidth}
        \centering
        \includegraphics[width= \columnwidth]{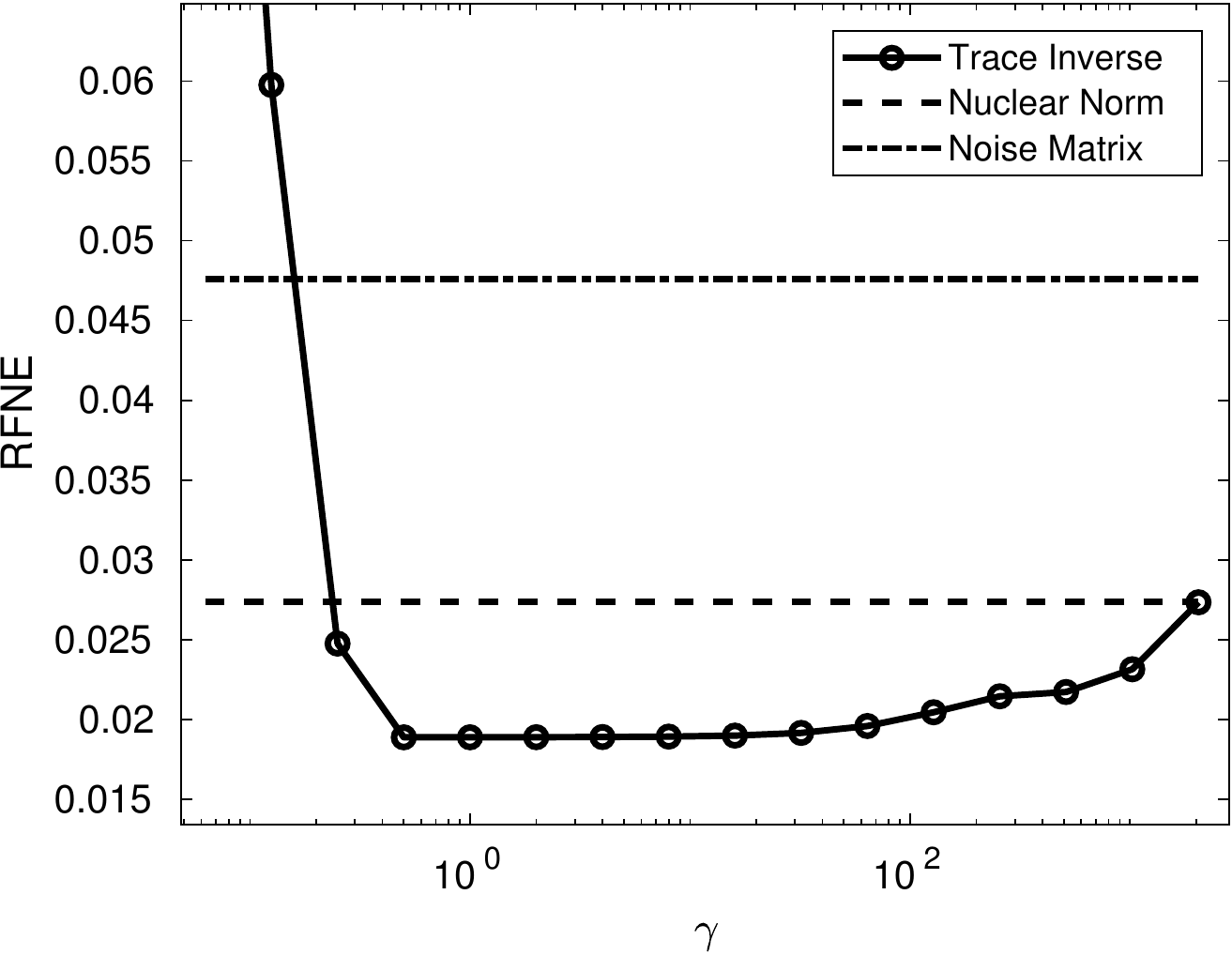}
        \caption{$m= 300, n= 200, r= 5$\\$ p=0.3 , d= 0.1$}
            \label{fig:gamma_smaller}
    \end{subfigure}%
    ~ 
    \begin{subfigure}[b]{0.45\textwidth}
        \centering
        \includegraphics[width= \columnwidth]{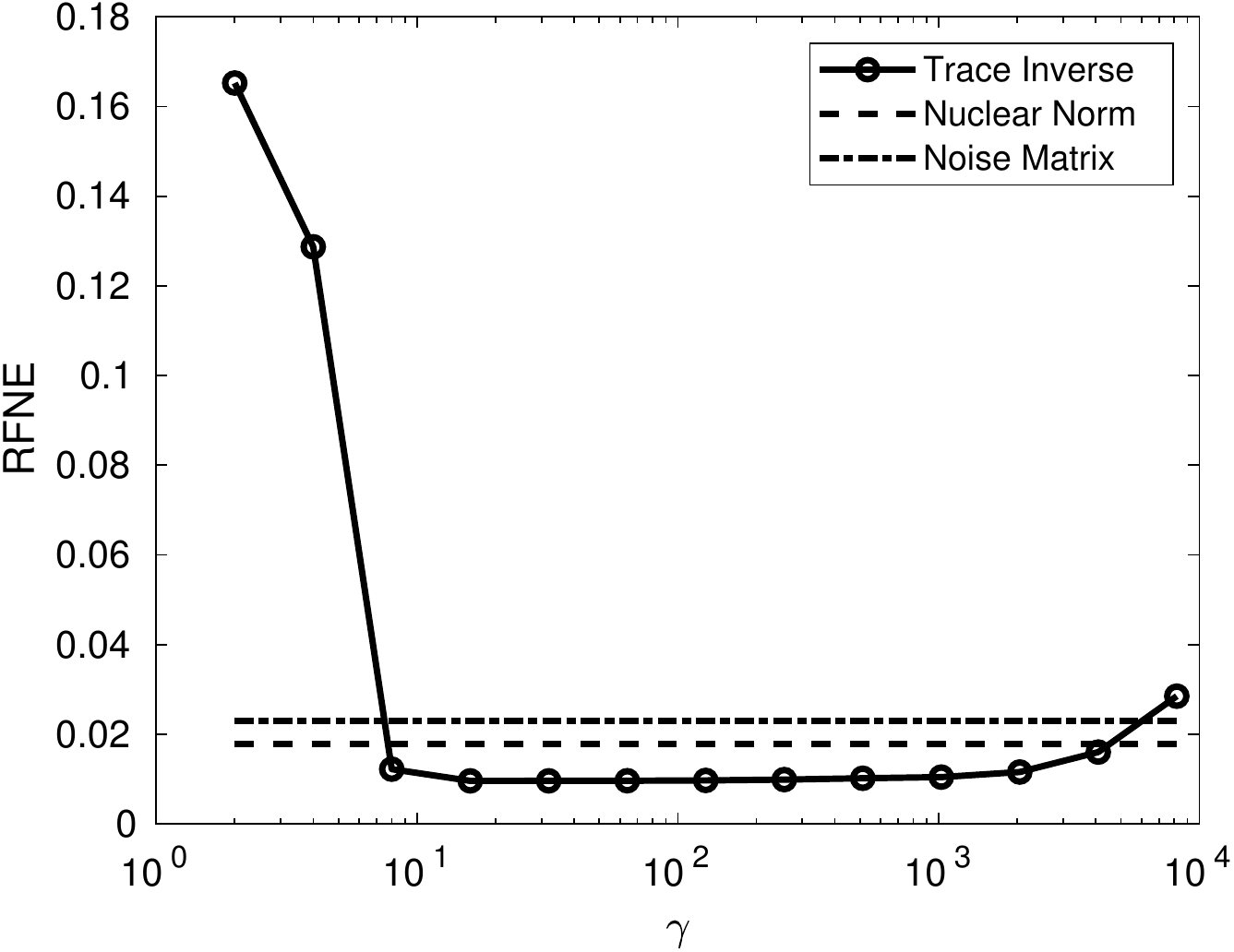}
        \caption{$m= 1000, n=500 , r= 5$\\$ p=0.1 , d= 0.05$}
            \label{fig:gamma_larger}
    \end{subfigure}
    \caption{RFNE of the matrix recovered by Algorithm \ref{alg:GenASD} utilizing the trace norm regularizer with values of $\gamma$ between $2^{-4}$ and $2^{8}$ in the left plot, and between $2^{1}$ and $2^{13}$ in the right plot, along with the RFNE of the noisy matrix and the optimal value to the nuclear norm minimization problem. }
        \label{fig:gamma}
\end{figure*}

Due to the remarkable consistency of the algorithm for varying choices of $\gamma$, parameter tuning is not an issue in practice.  Ideally, the choice of $\gamma$ would be approximately half of the largest nonzero singular value of the original low rank matrix so that the gradient of the regularizer is small for the top $r$ singular values.  While this quantity cannot be directly measured with incomplete noisy data, it can be (very roughly) approximated as follows:
$$\gamma= \frac{1}{2\sqrt{r p}} ||P_\Omega(\tilde{M})||_F $$
%\textbf{Note:} from the other theoretical work i'm doing, we discovered that there will be just one (low rank, incoherent) local optimum as long as $p \beta \gamma^2  \geq 6$ where $p$ is the percentage of known data.  This is a much better heuristic, however i don't want to publish that result yet. 
where $r$ is a rough estimate of the rank of the matrix. Note that, unlike rank constrained optimization methods which rely heavily on the rank of the matrix to be recovered being known exactly, Figure \ref{fig:gamma} indicates that our method will perform well even when the estimate of the rank is off by orders of magnitude.

\ifodd\value{verbose} {

Before moving on to larger, real data sets, we demonstrate the difference in speed between Algorithm \ref{alg:GenAltMin} when taking one step in the gradient direction and solving the $P_m$ and $P_n$ subproblems to optimality.  Figures \ref{fig:time_smaller} and \ref{fig:time_larger} plot the convergence of the two algorithms on matrices that are 300 by 200 and 1000 by 500 respectively.  First, note that in both figures the two methods converge to the same local optima, suggesting one need not worry about the difference in quality of the output between the two algorithms.

\begin{figure*}[t!]
    \centering
    \begin{subfigure}[b]{0.45\textwidth}
        \centering
        \includegraphics[width= \columnwidth]{300x200_all_algoriithms_loglog.png}
        \caption{$m=300 , n=200 , r=5$\\$ p=0.4 , d=0.05 $}
            \label{fig:time_smaller}
    \end{subfigure}%
    ~ 
    \begin{subfigure}[b]{0.45\textwidth}
        \centering
        \includegraphics[width= \columnwidth]{1000x500_all_algorithms_loglog.png}
        \caption{$m= 1000, n= 500, r= 15$\\$p=0.2 , d=0.01 $}
            \label{fig:time_larger}
    \end{subfigure}
    \caption{Convergence of Algorithm \ref{} (AM), Algorithm \ref{alg:GenAltMin} (SOR), and Algorithm \ref{alg:GenASD} (ASD). The RFNE and cumulative runtime is recorded at each iteration.}
\end{figure*}
}\else{

Before moving on to larger, real data sets, we demonstrate the difference in speed between Algorithm \ref{Agorthim:NoiseAltMin}  and Algorithm \ref{alg:GenASD}.  Figures \ref{fig:time_smaller} and \ref{fig:time_larger} plot the convergence of the two algorithms on matrices that are 300 by 200 and 1000 by 500 respectively.  First, note that in both figures the two methods converge to the same local optima, suggesting one need not worry about the difference in quality of the output between the two algorithms.

\begin{figure*}[t!]
    \centering
    \begin{subfigure}[b]{0.45\textwidth}
        \centering
        \includegraphics[width= \columnwidth]{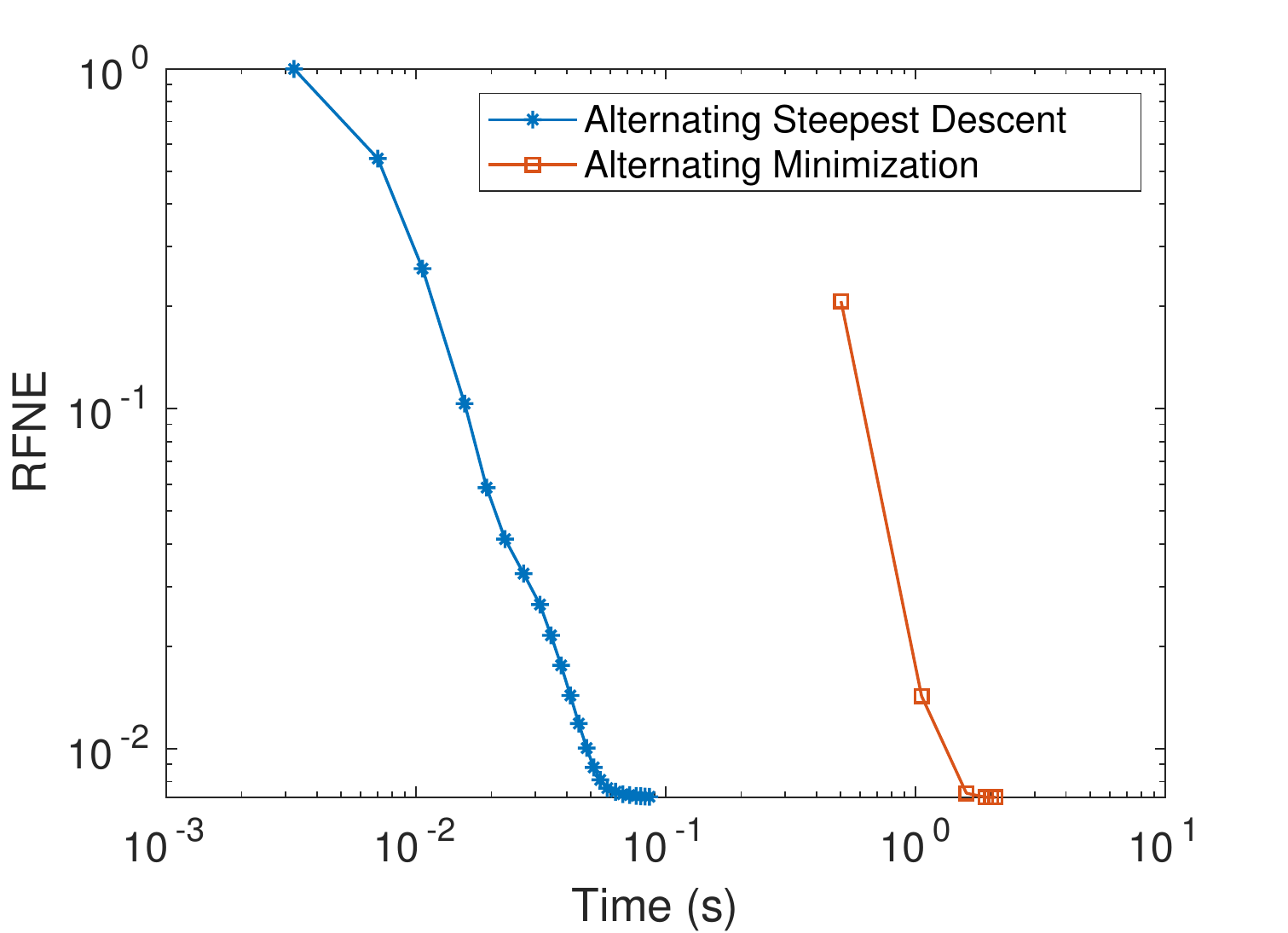}
        \caption{$m=300 , n=200 , r=5 $\\ $ p=0.4 , d=0.05 $}
            \label{fig:time_smaller}
    \end{subfigure}%
    ~ 
    \begin{subfigure}[b]{0.45\textwidth}
        \centering
        \includegraphics[width= \columnwidth]{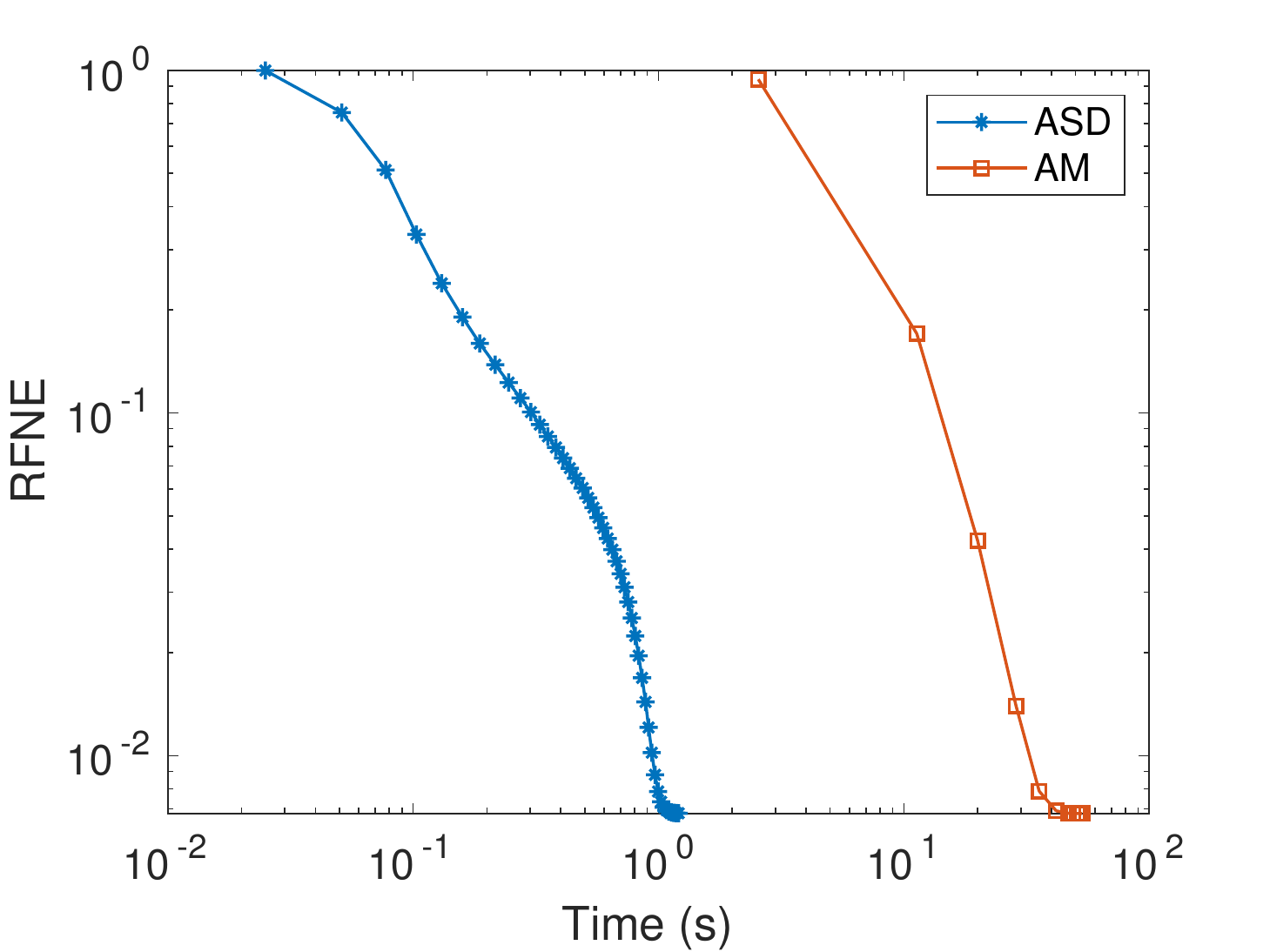}
        \caption{$m= 1000, n= 500, r= 15$\\$ p=0.2 , d=0.01 $}
            \label{fig:time_larger}
    \end{subfigure}
    \caption{Convergence of Algorithm \ref{Agorthim:NoiseAltMin} and Algorithm \ref{alg:GenASD}. The RFNE and cumulative runtime is recorded at each iteration.}
\end{figure*}

}\fi

For the smaller case, while clear that taking only one step converges faster than solving the subproblems to optimality, they both converge in under 2 seconds.  When solving the subproblems to optimality, however, only 4 iterations are needed to converge. In the larger case, the difference is much more apparent.  GenASD still converges in less than half of a second, where as solving the subproblems to optimality takes about 17 seconds.

We compare our algorithm to three other common matrix completion algorithms.  The algorithm presented by Yao et.\ al.\ \cite{yao_kwok2015}, Fast Nonconvex Low-Rank Matrix Learning (FaNCL), is the only other work we know of that solves \eqref{eqn:generalizedRelaxation} with iterations having computational complexity $\mathcal{O}(r|\Omega|)$.  The authors utilize nonconvex regularizers similar to the ones discussed in this paper, and use singular value thresholding with iteratively reweighted thresholds.  \revision{The FaNCL algorithm was later improved upon in \cite{yao_faster} by incorporating a momentum term for faster convergence.  We only compare to the earlier work as that was the code we had available.}

We also compare to FPC, which solved the nuclear norm minimization problem \cite{Ma2009}, and LMaFit, which solves the rank constrained problem \cite{Wen2012}.  Because LMaFit requires an estimate of the rank, we show results when the algorithm is given the correct rank and a rank twice as large as the original matrix to demonstrate the advantage of a rank minimization approach.
\begin{table} 
\caption{Comparison of four different matrix completion algorithms on randomly generated low rank matrices.  The algorithm LMaFit reconstructs a matrix of a given rank $k$.  The table shows the results when the algorithm is given the exact rank($k=r$) and an incorrect rank ($k=2r$). }
\centering
\scalebox{0.85}{\begin{tabular}{ccclllllll}
{\color[HTML]{000000} }                        & {\color[HTML]{000000} }                           & \multicolumn{1}{c|}{{\color[HTML]{000000} }}     & \multicolumn{2}{c|}{{\color[HTML]{000000} GenASD}}                                                                                                        & \multicolumn{2}{c|}{{\color[HTML]{000000} FaNCL}}                                                                                                                                                      & \multicolumn{1}{c|}{{\color[HTML]{000000} FPC}}    & \multicolumn{2}{c|}{{\color[HTML]{000000} LMaFit}}                                                   \\
\multicolumn{1}{c|}{{\color[HTML]{000000} r}}  & \multicolumn{1}{c|}{{\color[HTML]{000000} noise}} & \multicolumn{1}{c|}{{\color[HTML]{000000} p}}    & \multicolumn{1}{c|}{{\color[HTML]{000000} \begin{tabular}[c]{@{}c@{}}Trace \\ Inverse\end{tabular}}} & \multicolumn{1}{c|}{{\color[HTML]{000000} SCAD}}   & \multicolumn{1}{c|}{{\color[HTML]{000000} \begin{tabular}[c]{@{}c@{}}Log \\ Sum\end{tabular}}} & \multicolumn{1}{c|}{{\color[HTML]{000000} \begin{tabular}[c]{@{}c@{}}Capped \\ L1 Norm\end{tabular}}} & \multicolumn{1}{c|}{{\color[HTML]{000000} }}       & \multicolumn{1}{c|}{{\color[HTML]{000000} k=r}}    & \multicolumn{1}{c}{{\color[HTML]{000000} k=2r}} \\ \hline
\multicolumn{10}{c}{{\color[HTML]{000000} m=300, n=200}}                                                                                                                                                                                                                                                                                                                                                                                                                                                                                                                                                                                                                               \\ \hline
\multicolumn{1}{c|}{{\color[HTML]{000000} 5}}  & \multicolumn{1}{c|}{{\color[HTML]{000000} 0.05}}  & \multicolumn{1}{c|}{{\color[HTML]{000000} 0.1}}  & \multicolumn{1}{l|}{{\color[HTML]{000000} 0.0234}}                                                   & \multicolumn{1}{l|}{{\color[HTML]{000000} 0.0232}} & \multicolumn{1}{l|}{{\color[HTML]{000000} 0.0994}}                                             & \multicolumn{1}{l|}{{\color[HTML]{000000} 0.0546}}                                                    & \multicolumn{1}{l|}{{\color[HTML]{000000} 0.2374}} & \multicolumn{1}{l|}{{\color[HTML]{000000} 0.0273}} & {\color[HTML]{000000} 0.2892}                   \\ \hline
\multicolumn{1}{c|}{{\color[HTML]{000000} 5}}  & \multicolumn{1}{c|}{{\color[HTML]{000000} 0.05}}  & \multicolumn{1}{c|}{{\color[HTML]{000000} 0.3}}  & \multicolumn{1}{l|}{{\color[HTML]{000000} 0.0089}}                                                   & \multicolumn{1}{l|}{{\color[HTML]{000000} 0.0089}} & \multicolumn{1}{l|}{{\color[HTML]{000000} 0.0128}}                                             & \multicolumn{1}{l|}{{\color[HTML]{000000} 0.0092}}                                                    & \multicolumn{1}{l|}{{\color[HTML]{000000} 0.0171}} & \multicolumn{1}{l|}{{\color[HTML]{000000} 0.0089}} & {\color[HTML]{000000} 0.1035}                   \\ \hline
\multicolumn{1}{c|}{{\color[HTML]{000000} 5}}  & \multicolumn{1}{c|}{{\color[HTML]{000000} 0.1}}   & \multicolumn{1}{c|}{{\color[HTML]{000000} 0.1}}  & \multicolumn{1}{l|}{{\color[HTML]{000000} 0.0399}}                                                   & \multicolumn{1}{l|}{{\color[HTML]{000000} 0.0402}} & \multicolumn{1}{l|}{{\color[HTML]{000000} 0.0906}}                                             & \multicolumn{1}{l|}{{\color[HTML]{000000} 0.0476}}                                                    & \multicolumn{1}{l|}{{\color[HTML]{000000} 0.2573}} & \multicolumn{1}{l|}{{\color[HTML]{000000} 0.3616}} & {\color[HTML]{000000} 0.3027}                   \\ \hline
\multicolumn{1}{c|}{{\color[HTML]{000000} 5}}  & \multicolumn{1}{c|}{{\color[HTML]{000000} 0.1}}   & \multicolumn{1}{c|}{{\color[HTML]{000000} 0.3}}  & \multicolumn{1}{l|}{{\color[HTML]{000000} 0.018}}                                                    & \multicolumn{1}{l|}{{\color[HTML]{000000} 0.018}}  & \multicolumn{1}{l|}{{\color[HTML]{000000} 0.0203}}                                             & \multicolumn{1}{l|}{{\color[HTML]{000000} 0.0181}}                                                    & \multicolumn{1}{l|}{{\color[HTML]{000000} 0.0334}} & \multicolumn{1}{l|}{{\color[HTML]{000000} 0.018}}  & {\color[HTML]{000000} 0.1039}                   \\ \hline
\multicolumn{1}{c|}{{\color[HTML]{000000} 10}} & \multicolumn{1}{c|}{{\color[HTML]{000000} 0.05}}  & \multicolumn{1}{c|}{{\color[HTML]{000000} 0.1}}  & \multicolumn{1}{l|}{{\color[HTML]{000000} 0.7321}}                                                   & \multicolumn{1}{l|}{{\color[HTML]{000000} 0.0476}} & \multicolumn{1}{l|}{{\color[HTML]{000000} 0.2853}}                                             & \multicolumn{1}{l|}{{\color[HTML]{000000} 0.1742}}                                                    & \multicolumn{1}{l|}{{\color[HTML]{000000} 0.6683}} & \multicolumn{1}{l|}{{\color[HTML]{000000} 1.1706}} & {\color[HTML]{000000} 0.8913}                   \\ \hline
\multicolumn{1}{c|}{{\color[HTML]{000000} 10}} & \multicolumn{1}{c|}{{\color[HTML]{000000} 0.05}}  & \multicolumn{1}{c|}{{\color[HTML]{000000} 0.3}}  & \multicolumn{1}{l|}{{\color[HTML]{000000} 0.0094}}                                                   & \multicolumn{1}{l|}{{\color[HTML]{000000} 0.0093}} & \multicolumn{1}{l|}{{\color[HTML]{000000} 0.0156}}                                             & \multicolumn{1}{l|}{{\color[HTML]{000000} 0.0098}}                                                    & \multicolumn{1}{l|}{{\color[HTML]{000000} 0.0202}} & \multicolumn{1}{l|}{{\color[HTML]{000000} 0.0093}} & {\color[HTML]{000000} 0.1267}                   \\ \hline
\multicolumn{1}{c|}{{\color[HTML]{000000} 10}} & \multicolumn{1}{c|}{{\color[HTML]{000000} 0.1}}   & \multicolumn{1}{c|}{{\color[HTML]{000000} 0.1}}  & \multicolumn{1}{l|}{{\color[HTML]{000000} 0.7726}}                                                   & \multicolumn{1}{l|}{{\color[HTML]{000000} 0.091}}  & \multicolumn{1}{l|}{{\color[HTML]{000000} 0.3515}}                                             & \multicolumn{1}{l|}{{\color[HTML]{000000} 0.1942}}                                                    & \multicolumn{1}{l|}{{\color[HTML]{000000} 0.6349}} & \multicolumn{1}{l|}{{\color[HTML]{000000} 0.8075}} & {\color[HTML]{000000} 0.8429}                   \\ \hline
\multicolumn{1}{c|}{{\color[HTML]{000000} 10}} & \multicolumn{1}{c|}{{\color[HTML]{000000} 0.1}}   & \multicolumn{1}{c|}{{\color[HTML]{000000} 0.3}}  & \multicolumn{1}{l|}{{\color[HTML]{000000} 0.0193}}                                                   & \multicolumn{1}{l|}{{\color[HTML]{000000} 0.0193}} & \multicolumn{1}{l|}{{\color[HTML]{000000} 0.0233}}                                             & \multicolumn{1}{l|}{{\color[HTML]{000000} 0.0195}}                                                    & \multicolumn{1}{l|}{{\color[HTML]{000000} 0.0428}} & \multicolumn{1}{l|}{{\color[HTML]{000000} 0.0193}} & {\color[HTML]{000000} 0.1679}                   \\ \hline
\multicolumn{10}{c}{{\color[HTML]{000000} m=1000, n=500}}                                                                                                                                                                                                                                                                                                                                                                                                                                                                                                                                                                                                                              \\ \hline
\multicolumn{1}{c|}{{\color[HTML]{000000} 5}}  & \multicolumn{1}{c|}{{\color[HTML]{000000} 0.1}}   & \multicolumn{1}{c|}{{\color[HTML]{000000} 0.05}} & \multicolumn{1}{l|}{{\color[HTML]{000000} 0.031}}                                                    & \multicolumn{1}{l|}{{\color[HTML]{000000} 0.0311}} & \multicolumn{1}{l|}{{\color[HTML]{000000} 0.0493}}                                             & \multicolumn{1}{l|}{{\color[HTML]{000000} 0.039}}                                                     & \multicolumn{1}{l|}{{\color[HTML]{000000} 0.1436}} & \multicolumn{1}{l|}{{\color[HTML]{000000} 0.0314}} & {\color[HTML]{000000} 0.2074}                   \\ \hline
\multicolumn{1}{c|}{{\color[HTML]{000000} 5}}  & \multicolumn{1}{c|}{{\color[HTML]{000000} 0.1}}   & \multicolumn{1}{c|}{{\color[HTML]{000000} 0.1}}  & \multicolumn{1}{l|}{{\color[HTML]{000000} 0.0188}}                                                   & \multicolumn{1}{l|}{{\color[HTML]{000000} 0.0188}} & \multicolumn{1}{l|}{{\color[HTML]{000000} 0.023}}                                              & \multicolumn{1}{l|}{{\color[HTML]{000000} 0.0197}}                                                    & \multicolumn{1}{l|}{{\color[HTML]{000000} 0.0484}} & \multicolumn{1}{l|}{{\color[HTML]{000000} 0.0188}} & {\color[HTML]{000000} 0.1352}                   \\ \hline
\multicolumn{1}{c|}{{\color[HTML]{000000} 5}}  & \multicolumn{1}{c|}{{\color[HTML]{000000} 0.3}}   & \multicolumn{1}{c|}{{\color[HTML]{000000} 0.05}} & \multicolumn{1}{l|}{{\color[HTML]{000000} 0.1723}}                                                   & \multicolumn{1}{l|}{{\color[HTML]{000000} 0.0947}} & \multicolumn{1}{l|}{{\color[HTML]{000000} 0.1503}}                                             & \multicolumn{1}{l|}{{\color[HTML]{000000} 0.1216}}                                                    & \multicolumn{1}{l|}{{\color[HTML]{000000} 0.3023}} & \multicolumn{1}{l|}{{\color[HTML]{000000} 0.0943}} & {\color[HTML]{000000} 0.2946}                   \\ \hline
\multicolumn{1}{c|}{{\color[HTML]{000000} 5}}  & \multicolumn{1}{c|}{{\color[HTML]{000000} 0.3}}   & \multicolumn{1}{c|}{{\color[HTML]{000000} 0.1}}  & \multicolumn{1}{l|}{{\color[HTML]{000000} 0.0994}}                                                   & \multicolumn{1}{l|}{{\color[HTML]{000000} 0.0582}} & \multicolumn{1}{l|}{{\color[HTML]{000000} 0.0894}}                                             & \multicolumn{1}{l|}{{\color[HTML]{000000} 0.057}}                                                     & \multicolumn{1}{l|}{{\color[HTML]{000000} 0.1061}} & \multicolumn{1}{l|}{{\color[HTML]{000000} 0.0566}} & {\color[HTML]{000000} 0.1333}                   \\ \hline
\multicolumn{1}{c|}{{\color[HTML]{000000} 10}} & \multicolumn{1}{c|}{{\color[HTML]{000000} 0.1}}   & \multicolumn{1}{c|}{{\color[HTML]{000000} 0.05}} & \multicolumn{1}{l|}{{\color[HTML]{000000} 0.8952}}                                                   & \multicolumn{1}{l|}{{\color[HTML]{000000} 0.0424}} & \multicolumn{1}{l|}{{\color[HTML]{000000} 0.4071}}                                             & \multicolumn{1}{l|}{{\color[HTML]{000000} 0.0795}}                                                    & \multicolumn{1}{l|}{{\color[HTML]{000000} 0.5504}} & \multicolumn{1}{l|}{{\color[HTML]{000000} 0.0475}} & {\color[HTML]{000000} 0.6989}                   \\ \hline
\multicolumn{1}{c|}{{\color[HTML]{000000} 10}} & \multicolumn{1}{c|}{{\color[HTML]{000000} 0.1}}   & \multicolumn{1}{c|}{{\color[HTML]{000000} 0.1}}  & \multicolumn{1}{l|}{{\color[HTML]{000000} 0.0207}}                                                   & \multicolumn{1}{l|}{{\color[HTML]{000000} 0.0207}} & \multicolumn{1}{l|}{{\color[HTML]{000000} 0.0273}}                                             & \multicolumn{1}{l|}{{\color[HTML]{000000} 0.0222}}                                                    & \multicolumn{1}{l|}{{\color[HTML]{000000} 0.0684}} & \multicolumn{1}{l|}{{\color[HTML]{000000} 0.0208}} & {\color[HTML]{000000} 0.1612}                   \\ \hline
\multicolumn{1}{c|}{{\color[HTML]{000000} 10}} & \multicolumn{1}{c|}{{\color[HTML]{000000} 0.3}}   & \multicolumn{1}{c|}{{\color[HTML]{000000} 0.05}} & \multicolumn{1}{l|}{{\color[HTML]{000000} 0.8675}}                                                   & \multicolumn{1}{l|}{{\color[HTML]{000000} 0.1231}} & \multicolumn{1}{l|}{{\color[HTML]{000000} 0.5028}}                                             & \multicolumn{1}{l|}{{\color[HTML]{000000} 0.1544}}                                                    & \multicolumn{1}{l|}{{\color[HTML]{000000} 0.5626}} & \multicolumn{1}{l|}{{\color[HTML]{000000} 0.1258}} & {\color[HTML]{000000} 0.7173}                   \\ \hline
\multicolumn{1}{c|}{{\color[HTML]{000000} 10}} & \multicolumn{1}{c|}{{\color[HTML]{000000} 0.3}}   & \multicolumn{1}{c|}{{\color[HTML]{000000} 0.1}}  & \multicolumn{1}{l|}{{\color[HTML]{000000} 0.1139}}                                                   & \multicolumn{1}{l|}{{\color[HTML]{000000} 0.0623}} & \multicolumn{1}{l|}{{\color[HTML]{000000} 0.1013}}                                             & \multicolumn{1}{l|}{{\color[HTML]{000000} 0.063}}                                                     & \multicolumn{1}{l|}{{\color[HTML]{000000} 0.1563}} & \multicolumn{1}{l|}{{\color[HTML]{000000} 0.0622}} & {\color[HTML]{000000} 0.2003}                  
\end{tabular}
}
\end{table}

With minor exceptions, the algorithm presented in this paper, FaNCL, and LMaFit when given the correct rank all give approximately the same quality result.  GenAltMin solves the problem faster than FaNCL in every case. Although GenAltMin and FaNCL take approximately the same amount of time per iteration, singular value thresholding methods take significantly more iterations. Our algorithm outperforms FPC for reasons discussed earlier in this section, and also LMaFit when the rank is not well known.  

\subsection{Collaborative Filtering}
Perhaps the most widely known application of rank minimization is the Netflix Problem, wherein the goal is to predict how a user would rate a movie based on how she rated other movies, along with how other users with similar taste rated said movie. To formulate this as a matrix completion problem, we have a sparse matrix whose columns correspond to different movies and whose rows correspond to different users, with the entries of the matrix being how a user rated a specific movie. We expect that if every entry of this matrix was observed, the matrix would be low rank because the number of factors contributing to how much someone enjoys a movie is far less than the total number of movies or users in the data set.

\begin{table} 
\caption{NMAE utilizing Algorithm \ref{alg:GenASD} with the trace inverse regularizer and with the nuclear norm regularizer, along with LMaFit}
\label{table:ML_MC}
\centering
\scalebox{0.85}{
\begin{tabular}{l|l|l|l|l|l|l|l|l|l}
{\color[HTML]{000000} }     & \multicolumn{3}{l|}{{\color[HTML]{000000} MovieLens100k}}                                              & \multicolumn{3}{l|}{{\color[HTML]{000000} MovieLens1m}}                                                & \multicolumn{3}{l|}{{\color[HTML]{000000} Jester}}                                                     \\ \hline
{\color[HTML]{000000} Fold} & {\color[HTML]{000000} TI}              & {\color[HTML]{000000} NN}     & {\color[HTML]{000000} LmaFit} & {\color[HTML]{000000} TI}              & {\color[HTML]{000000} NN}     & {\color[HTML]{000000} LmaFit} & {\color[HTML]{000000} TI}              & {\color[HTML]{000000} NN}     & {\color[HTML]{000000} LmaFit} \\ \hline
{\color[HTML]{000000} 1}    & {\color[HTML]{000000} \textbf{0.1724}} & {\color[HTML]{000000} 0.1812} & {\color[HTML]{000000} 0.1800}   & {\color[HTML]{000000} \textbf{0.1683}} & {\color[HTML]{000000} 0.1695} & {\color[HTML]{000000} 0.1820}  & {\color[HTML]{000000} \textbf{0.1570}}  & {\color[HTML]{000000} 0.1607} & {\color[HTML]{000000} 0.1600}   \\ \hline
{\color[HTML]{000000} 2}    & {\color[HTML]{000000} \textbf{0.1719}} & {\color[HTML]{000000} 0.1799} & {\color[HTML]{000000} 0.1775} & {\color[HTML]{000000} \textbf{0.1676}} & {\color[HTML]{000000} 0.1699} & {\color[HTML]{000000} 0.1811} & {\color[HTML]{000000} \textbf{0.1577}} & {\color[HTML]{000000} 0.1610}  & {\color[HTML]{000000} 0.1601} \\ \hline
{\color[HTML]{000000} 3}    & {\color[HTML]{000000} \textbf{0.1702}} & {\color[HTML]{000000} 0.1785} & {\color[HTML]{000000} 0.1781} & {\color[HTML]{000000} \textbf{0.1682}} & {\color[HTML]{000000} 0.1695} & {\color[HTML]{000000} 0.1825} & {\color[HTML]{000000} \textbf{0.1572}} & {\color[HTML]{000000} 0.1604} & {\color[HTML]{000000} 0.1596} \\ \hline
{\color[HTML]{000000} 4}    & {\color[HTML]{000000} \textbf{0.1715}} & {\color[HTML]{000000} 0.1789} & {\color[HTML]{000000} 0.1787} & {\color[HTML]{000000} \textbf{0.1685}} & {\color[HTML]{000000} 0.1703} & {\color[HTML]{000000} 0.1824} & {\color[HTML]{000000} \textbf{0.1572}} & {\color[HTML]{000000} 0.1603} & {\color[HTML]{000000} 0.1602} \\ \hline
{\color[HTML]{000000} 5}    & {\color[HTML]{000000} \textbf{0.1732}} & {\color[HTML]{000000} 0.1822} & {\color[HTML]{000000} 0.1788} & {\color[HTML]{000000} \textbf{0.1678}} & {\color[HTML]{000000} 0.1691} & {\color[HTML]{000000} 0.1815} & {\color[HTML]{000000} \textbf{0.1574}} & {\color[HTML]{000000} 0.1612} & {\color[HTML]{000000} 0.1601} \\ \hline
{\color[HTML]{000000} avg}  & {\color[HTML]{000000} \textbf{0.1719}} & {\color[HTML]{000000} 0.1802} & {\color[HTML]{000000} 0.1786} & {\color[HTML]{000000} \textbf{0.1681}} & {\color[HTML]{000000} 0.1697} & {\color[HTML]{000000} 0.1819} & {\color[HTML]{000000} \textbf{0.1573}} & {\color[HTML]{000000} 0.1607} & {\color[HTML]{000000} 0.1600}  
\end{tabular}}
\end{table}

We utilize Algorithm 4.2 and LMaFit on the MovieLens100k and MovieLens1m datasets \cite{movielens}, and the Jester dataset \cite{Jester}.  Both MovieLens datasets consist of ratings on various movies, rated from 1 to 5, and the Jester dataset consists of ratings on jokes, rated -10 to 10.  The MovieLens100k dataset has 1,000 users, 1,700 movies, and 100,000 measurements, the MovieLens1m dataset has 6,000 users, 4,000 movies, and 1 million measurements, and the Jester dataset has 24,983 users, 101 jokes, and 689,000 measurements.  Note that while the movie lens datasets are both very sparse (approximately 5\%), the Jester dataset has 27\% of all possible ratings.  
	
For each dataset, we separate the data into five partitions, and for each partition we use the remaining four partitions to find a low rank matrix, and the fifth partition to test our results.  In Table \ref{table:ML_MC}, we report the normalized mean absolute error (NMAE), defined as
$$\text{NMAE}=\frac{1}{n_\text{ratings}} \sum_i \frac{|y_i - \tilde{y}_i|}{y_\text{max}-y_\text{min}}$$
where $n_{\text{ratings}}$ is the total number of ratings used in the testing set, $y$ is the measurements from the dataset, $\tilde{y}$ are the predictions from the low rank matrix, and $y_\text{max}$ and $y_\text{min}$ are the maximum and minimum ratings for the dataset (5 and 1 for the MovieLens dataset, and -10 and 10 for the Jester dataset).  In each case, we use 10 as the upper bound on the rank. We found that the NMAE for LMaFit is minimized when constrained to a rank 1 matrix, which is what is reported.  
	
In every fold in each of the three datasets, Algorithm 4.2 utilizing the trace norm regularizer outperforms the nuclear norm regularizer and LMaFit.  To gain insight as to why the trace inverse regularizer outperforms the other methods, we examine the singular value distribution of the resulting low rank matrix.  The singular values for the matrices recovered from the MovieLens1M dataset withholding fold 5 is shown for each method in Figure \ref{fig:movie_lens_svd}.  Comparing the trace inverse to the nuclear norm, the first singular value of the matrix recovered with the trace inverse regularizer is larger, and the rest are smaller, which is expected because the trace inverse puts more weight on minimizing smaller singular value and less weight on minimizing larger singular values.  Because the ratings matrix is close to a rank one matrix, penalizing the largest singular value is disadvantageous because we expect it to be large.   Additionally, as opposed to the result from LMaFit, the remaining 9 singular values are nonzero.  This demonstrates the advantage of rank minimization methods over rank constrained methods:  while we may want to put more emphasis on the first singular value, the remaining singular values are still important.  In a rank constrained paradigm, there is no way to both keep singular values and also minimize them. 

\begin{figure}
    \centering
\includegraphics[width=0.8\columnwidth]{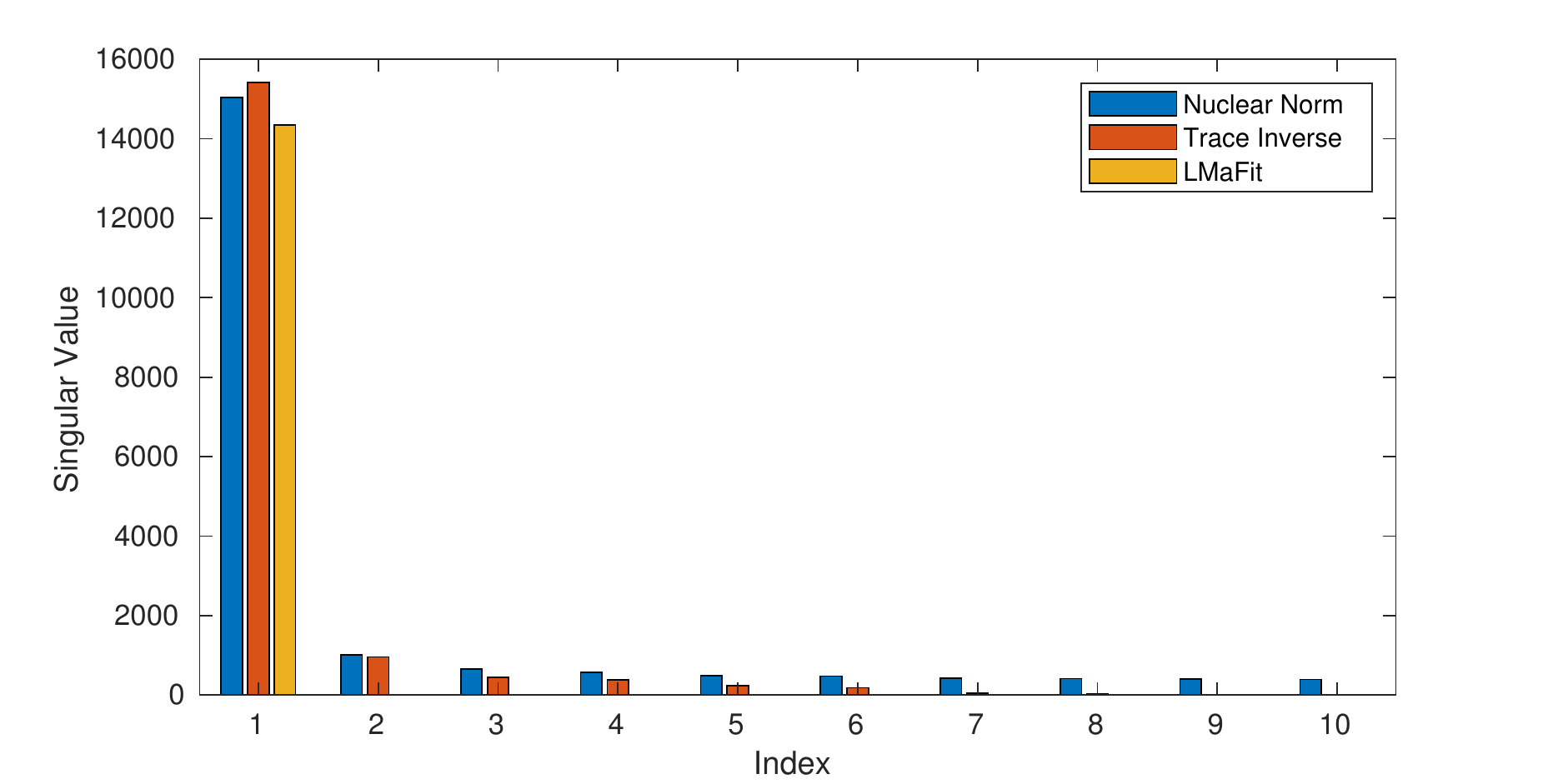}
    \caption{Singular value decomposition for the matrix recovered from the MovieLens1M dataset withholding fold 5.}
    \label{fig:movie_lens_svd}
\end{figure}

\section{Conclusions}
We have shown that the problem of minimizing the rank of a matrix using nonconvex regularizers can be posed as a bi-convex semidefinite optimization problem. By doing so, we were able to derive efficient algorithms using a low rank factorization and show convergence.

The methods are shown to be computationally superior to methods based off of the nuclear norm relaxation, and that the estimator bias is drastically reduced by using nonconvex regularizers.  We show that the quality of the result from our algorithm hardly changes when either of the parameters are changed by multiple orders of magnitude.  Additionally, we show that our method is faster than other existing methods based off of nonconvex regularizers.

\bibliographystyle{spmpsci}      % mathematics and physical sciences

\bibliography{references}   % name your BibTeX data base

\end{document}